\documentclass{article}

\usepackage{amsmath,amssymb,amsthm,mathrsfs}
\usepackage{graphicx}
\usepackage[colorlinks=true]{hyperref}
\usepackage{pdfsync}

\newtheorem{theorem}{Theorem}
\newtheorem{proposition}{Proposition}
\newtheorem{lemma}{Lemma}

\theoremstyle{definition}\newtheorem{example}{Example}
\theoremstyle{definition}\newtheorem{definition}{Definition}
\theoremstyle{definition}\newtheorem{remark}{Remark}

\def\R{\textrm{I\kern-0.21emR}}
\def\N{\textrm{I\kern-0.21emN}}
\def\Z{\mathbb{Z}}
\def\T{\mathbb{T}}
\def\TS{\T \backslash \{ \sup \T \}}

\def\B{\overline{B}}
\def\L{\mathrm{L}}
\def\E{\mathrm{E}}
\def\S{\mathrm{S}}
\def\D{\mathcal{D}^\Omega}
\def\DS{\mathcal{D}^\Omega_\mathrm{stab}}
\def\V{\mathcal{V}}
\def\U{\mathcal{U}}
\def\UU{\mathcal{UQ}^b_{\mathrm{ad}}}

\def\O{\mathcal{O}}
\def\P{\mathrm{P}}
\def\PPP{\mathrm{Aff}(\Omega)}

\def\CC{\mathscr{C}}
\def\RR{\mathrm{RS}}
\def\SS{\mathrm{RD}}
\def\LL{\mathscr{L}}
\def\DD{\Delta}
\def\AC{\mathrm{AC}}
\def\Int{\mathrm{Int}}
\def\Coad{\mathrm{\overline{Co}}}
\def\Opp{\mathrm{Opp}}
\def\OCP{\bf (OCP)_\T}

\DeclareMathOperator*{\supess}{sup\,ess}

\newcommand{\fonction}[5]{\begin{array}[t]{lrcl}#1 :&#2 &\longrightarrow &#3\\&#4& \longmapsto &#5 \end{array}}

\renewcommand{\leq}{\leqslant}
\renewcommand{\geq}{\geqslant}

\title{Pontryagin Maximum Principle for finite dimensional nonlinear optimal control problems on time scales}

\author{Lo\"ic Bourdin\footnote{Laboratoire de Math\'ematiques et de leurs Applications - Pau (LMAP). UMR CNRS 5142. Universit\'e de Pau et des Pays de l'Adour. \texttt{bourdin.l@univ-pau.fr}} ,
Emmanuel Tr\'elat\footnote{Universit\'e Pierre et Marie Curie (Univ. Paris 6) and Institut Universitaire de France, CNRS UMR 7598, Laboratoire Jacques-Louis Lions, F-75005, Paris, France. \texttt{emmanuel.trelat@upmc.fr}}
}

\date{}

\begin{document}

\maketitle

\begin{abstract}
In this article we derive a strong version of the Pontryagin Maximum Principle for general nonlinear optimal control problems on time scales in finite dimension. The final time can be fixed or not, and in the case of general boundary conditions we derive the corresponding transversality conditions. Our proof is based on Ekeland's variational principle. Our statement and comments clearly show the distinction between right-dense points and right-scattered points. At right-dense points a maximization condition of the Hamiltonian is derived, similarly to the continuous-time case. At right-scattered points a weaker condition is derived, in terms of so-called stable $\Omega$-dense directions. We do not make any specific restrictive assumption on the dynamics or on the set $\Omega$ of control constraints.
Our statement encompasses the classical continuous-time and discrete-time versions of the Pontryagin Maximum Principle, and holds on any general time scale, that is any closed subset of $\R$.
\end{abstract}

\bigskip


\bigskip


\section{Introduction}\label{part0}
Optimal control theory is concerned with the analysis of controlled dynamical systems, where one aims at steering such a system from a given configuration to some desired target one by minimizing or maximizing some criterion. The Pontryagin Maximum Principle (denoted in short PMP), established at the end of the fifties for finite dimensional general nonlinear continuous-time dynamics (see \cite{pont}, and see \cite{gamk} for the history of this discovery), is the milestone of the classical optimal control theory. It provides a first-order necessary condition for optimality, by asserting that any optimal trajectory must be the projection of an extremal. The PMP then reduces the search of optimal trajectories to a boundary value problem posed on extremals. Optimal control theory, and in particular the PMP, have an immense field of applications in various domains, and it is not our aim here to list them. We refer the reader to textbooks on optimal control such as
\cite{agrach,Bon-Chy03,trel2,bres,brys,BulloLewis,hest,Jurdjevic,lee,pont,Schattler,seth,trel} 
for many examples of theoretical or practical applications of optimal control, essentially in a continuous-time setting.

Right after this discovery the corresponding theory has been developed for discrete-time dynamics, under appropriate convexity assumptions (see e.g.\ \cite{Halkin,holt2,holt}), leading to a version of the PMP for discrete-time optimal control problems. The considerable development of the discrete-time control theory was motivated by many potential applications e.g.\ to digital systems or in view of discrete approximations in numerical simulations of differential controlled systems. We refer the reader to the  textbooks \cite{bolt,cano,mord,seth} for details on this theory and many examples of applications.
It can be noted that some early works devoted to the discrete-time PMP (like \cite{fan}) are mathematically incorrect. Many counter-examples were provided in \cite{bolt} (see also \cite{mord}), showing that, as is now well known, the exact analogous of the continuous-time PMP does not hold at the discrete level. More precisely, the maximization condition of the PMP cannot be expected to hold in general in the discrete-time case. Nevertheless a weaker condition can be derived, see \cite[Theorem~42.1 p.\ 330]{bolt}. Note as well that approximate maximization conditions are given in \cite[Section~6.4]{mord} and that a wide literature is devoted to the introduction of convexity assumptions on the dynamics allowing one to recover the maximization condition in the discrete case (such as the concept of \textit{directional convexity} assumption used in \cite{cano,holt2,holt} for example).

The \textit{time scale} theory was introduced in \cite{hilg} in order to unify discrete and continuous analysis. A time scale $\T$ is an arbitrary non empty closed subset of $\R$, and a dynamical system is said to be posed on the time scale $\T$ whenever the time variable evolves along this set $\T$. The continuous-time case corresponds to $\T=\R$ and the discrete-time case corresponds to $\T=\Z$. 
The time scale theory aims at closing the gap between continuous and discrete cases and allows one to treat more general models of processes involving both continuous and discrete time elements, and more generally for dynamical systems where the time evolves along a set of a complex nature which may even be a Cantor set (see e.g.\ \cite{gama,may} for a study of a seasonally breeding population whose generations do not overlap, or \cite{ati} for applications to economics). 
Many notions of standard calculus have been extended to the time scale framework, and we refer the reader to \cite{agar2,agar3,bohn,bohn3} for details on this theory.
 
The theory of the calculus of variations on time scales, initiated in \cite{bohn2}, has been well studied in the existing literature (see e.g.\ \cite{torr8,bohn4,torr9,hils,hils1}).
Few attempts have been made to derive a PMP on time scales. In \cite{hils2} the authors establish a \textit{weak} PMP for shifted controlled systems, where the controls are not subject to any pointwise constraint and under certain restrictive assumptions.
A strong version of the PMP is claimed in \cite{zhan2} but many arguments thereof are erroneous (see Remark \ref{remZhan} for details). 

The objective of the present article is to state and prove a strong version of the PMP on time scales, valuable for general nonlinear dynamics, and without assuming any unnecessary Lipschitz or convexity conditions. Our statement is as general as possible, and encompasses the classical continuous-time PMP that can be found e.g.\ in \cite{lee,pont} as well as all versions of discrete-time PMP's mentioned above. In accordance with all known results, the maximization condition is obtained at right-dense points of the time scale and a weaker one (similar to \cite[Theorem 42.1 p.\ 330]{bolt}) is derived at right-scattered points. Moreover, we consider general constraints on the initial and final values of the state variable and we derive the resulting transversality conditions. We provide as well a version of the PMP for optimal control problems with parameters.

The article is structured as follows. In Section~\ref{part1}, we first provide some basic issues of time scale calculus (Subsection~\ref{section1}). We define some appropriate notions such as the notion of stable $\Omega$-dense direction in Subsection \ref{sec_topoprelim}. In Subsection \ref{section2} we settle the notion of admissible control and define general optimal control problems on time scales. Our main result (Pontryagin Maximum Principle, Theorem \ref{thmmain}) is stated in Subsection \ref{section2bis1}, and we analyze and comment the results in a series of remarks.
Section \ref{part3} is devoted to the proof of Theorem \ref{thmmain}. First, in Subsection \ref{section10} we make some preliminary comments explaining which obstructions may appear when dealing with general time scales, and why we were led to a proof based on Ekeland's Variational Principle. We also comment on the article \cite{zhan2} in Remark \ref{remZhan}. In Subsection \ref{section3}, after having shown that the set of admissible controls is open, we define needle-like variations at right-dense and right-scattered points and derive some properties. In Subsection \ref{section4}, we apply Ekeland's Variational Principle to a well chosen functional in an appropriate complete metric space and then prove the PMP.

\section{Main result}\label{part1}
Let $\T$ be a time scale, that is, an arbitrary closed subset of $\R$. We assume throughout that $\T$ is bounded below and that $\mathrm{card} (\T) \geq 2$. We denote by $a = \min \T$. 

\subsection{Preliminaries on time scales}\label{section1}
For every subset $A$ of $\R$, we denote by $A_\T = A \cap \T$. An interval of $\T$ is defined by $I_\T$ where $I$ is an interval of $\R$.

The backward and forward jump operators $\rho,\sigma:\T\rightarrow\T$ are respectively defined by $\rho (t) = \sup \{ s \in \T\ \vert\ s < t \}$ and $\sigma (t) = \inf \{ s \in \T\ \vert\ s > t \}$ for every $t \in \T$, where $\rho (a) = a$ and $\sigma(\max \T) = \max \T$ whenever $\T$ is bounded above. A point $t \in \T$ is said to be \textit{left-dense} (respectively, \textit{left-scattered}, \textit{right-dense} or \textit{right-scattered}) whenever $\rho (t) = t$ (respectively, $\rho (t) < t$, $\sigma (t) = t$ or $\sigma (t) > t$). The graininess function $\mu:\T\rightarrow\R^+$ is defined by $\mu(t) = \sigma (t) -t$ for every $t \in \T$. 
We denote by $\RR$ the set of all right-scattered points of $\T$, and by $\SS$ the set of all right-dense points of $\T$ in $\T \backslash \{ \sup \T \}$. 
Note that $\RR$ is a subset of $ \T \backslash \{ \sup \T \}$, is at most countable (see \cite[Lemma 3.1]{caba}), and note that $\SS$ is the complement of $\RR$ in $\T \backslash \{ \sup \T \}$. For every $b \in \T \backslash \{ a \}$ and every $s \in [a,b[_\T \cap \SS$, we set
\begin{equation}\label{defVbs}
\V^b_s = \{ \beta \geq 0, \; s+\beta \in [s,b]_\T \}.
\end{equation}
Note that $0$ is not isolated in $\V^b_s$.

\paragraph{$\DD$-differentiability.}
We set $\T^\kappa = \T \backslash \{ \max \T \}$ whenever $\T$ admits a left-scattered maximum, and $\T^\kappa = \T$ otherwise. Let $n \in \N^*$; a function $q:\T\rightarrow\R^n$ is said to be $\DD$-differentiable at $t \in \T^\kappa$ if the limit
$$q^\DD (t) = \lim\limits_{\substack{s \to t \\ s \in \T}} \frac{q^\sigma (t) -q(s)}{\sigma (t) -s} $$
exists in $\R^n$, where $q^\sigma = q \circ \sigma$. Recall that, if $t \in \T^\kappa$ is a right-dense point of $\T$, then $q$ is $\DD$-differentiable at $t$ if and only if the limit of $\frac{q(t)-q(s)}{t-s}$ as $s \to t$, $s \in \T$, exists; in that case it is equal to $q^\DD (t)$.
If $t \in \T^\kappa$ is a right-scattered point of $\T$ and if $q$ is continuous at $t$, then $q$ is $\DD$-differentiable at $t$, and $q^\DD (t) = (q^\sigma(t) - q(t)) / \mu(t)$ (see \cite{bohn}).

If $q,q':\T\rightarrow\R^n$ are both $\DD$-differentiable at $t \in \T^\kappa$, then the scalar product $\langle q, q' \rangle_{\R^n}$ is $\DD$-differentiable at $t$ and 
\begin{equation}\label{eqleibniz}
\langle q, q' \rangle_{\R^n}^\DD (t) = \langle q^\DD (t), q'^\sigma(t) \rangle_{\R^n}  + \langle q (t), q'^\DD(t) \rangle_{\R^n} = \langle q^\DD (t), q'(t) \rangle_{\R^n}  + \langle q^\sigma (t), q'^\DD(t) \rangle_{\R^n}
\end{equation}
(Leibniz formula, see \cite[Theorem 1.20]{bohn}).

\paragraph{Lebesgue $\DD$-measure and Lebesgue $\DD$-integrability.}
Let $\mu_\DD$ be the Lebesgue $\DD$-measure on $\T$ defined in terms of Carath\'eodory extension in \cite[Chapter 5]{bohn3}. We also refer the reader to \cite{agar,caba,guse} for more details on the $\mu_\DD$-measure theory. For all $(c,d)\in\T^2$ such that $c \leq d $, there holds $\mu_\DD ([c,d[_\T) = d-c$. Recall that $A \subset \T$ is a $\mu_\DD$-measurable set of $\T$ if and only if $A$ is an usual $\mu_L$-measurable set of $\R$, where $\mu_L$ denotes the usual Lebesgue measure (see \cite[Proposition 3.1]{caba}). Moreover, if $A \subset \T \backslash \{ \sup \T \}$, then 
$$\mu_\DD ( A ) = \mu_L (A) +  \sum_{r \in A \cap \RR} \mu (r).$$
Let $A \subset \T$. A property is said to hold $\DD$-almost everywhere (in short, $\DD$-a.e.) on $A$ if it holds for every $t \in A \backslash A'$, where $A' \subset A$ is some $\mu_\DD$-measurable subset of $\T$ satisfying $\mu_\DD (A') = 0$. In particular, since $\mu_\DD (\{ r \}) = \mu (r) > 0$ for every $r \in \RR$, we conclude that if a property holds $\DD$-a.e.\ on $A$, then it holds for every $r \in A \cap \RR$. 

Let $n \in \N^*$ and let $A \subset \T \backslash \{ \sup \T \}$ be a $\mu_\DD$-measurable subset of $\T$. Consider a function $q$ defined $\DD$-a.e.\ on $A$ with values in $\R^n$. Let $A_0=A \cup ]r,\sigma(r)[_{r \in A \cap \RR}$, and let $q_0$ be the extension of $q$ defined $\mu_L$-a.e.\ on $A_0$ by $q_0 (t)=q(t)$ whenever $t \in A$, and by $q(t)=q(r)$ whenever $t \in ]r,\sigma(r)[$, for every $r \in A \cap \RR$.
Recall that $q$ is $\mu_\DD$-measurable on $A$ if and only if $q_0$ is $\mu_L$-measurable on $A_0$ (see \cite[Proposition 4.1]{caba}). 

The functional space $\L^\infty_\T (A,\R^n)$ is the set of all functions $q$ defined $\DD$-a.e.\ on $A$, with values in $\R^n$, that are $\mu_\DD$-measurable on $A$ and bounded almost everywhere.
Endowed with the norm $\Vert q \Vert_{\L^\infty_\T (A,\R^n)} = \supess\limits_{\tau \in A} \Vert q(\tau) \Vert_{\R^n}$, it is a Banach space (see \cite[Theorem 2.5]{agar}). Here the notation $\Vert\ \Vert_{\R^n}$ stands for the usual Euclidean norm of $\R^n$.
The functional space $\L^1_\T (A,\R^n)$ is the set of all functions $q$ defined $\DD$-a.e.\ on $A$, with values in $\R^n$, that are $\mu_\DD$-measurable on $A$ and such that
$\int_{A} \Vert q(\tau) \Vert_{\R^n} \, \DD \tau < + \infty$.
Endowed with the norm $\Vert q \Vert_{\L^1_\T (A,\R^n)} = \int_{A} \Vert q(\tau) \Vert_{\R^n} \, \DD \tau$, it is a Banach space (see \cite[Theorem 2.5]{agar}).
We recall here that if $q \in \L^1_\T (A,\R^n)$ then
$$ \int_{A} q(\tau) \, \DD \tau = \int_{A_0} q_0(\tau) \, d\tau =  \int_{A} q(\tau) \, d\tau +  \sum_{r \in A \cap \RR} \mu (r) q(r) $$
(see \cite[Theorems 5.1 and 5.2]{caba}). Note that if $A$ is bounded then $\L^\infty_\T (A,\R^n) \subset \L^1_\T (A,\R^n)$. 
The functional space $\L^\infty_{\mathrm{loc},\T} (\TS,\R^n)$ is the set of all functions $q$ defined $\DD$-a.e.\ on $\TS$, with values in $\R^n$, that are $\mu_\DD$-measurable on $\TS$ and such that $q \in \L^\infty_\T ([c,d[_\T,\R^n)$ for all $(c,d)\in\T^2$ such that $c < d$.

\paragraph{Absolutely continuous functions.}
Let $n \in \N^*$ and let $(c,d) \in \T^2$ such that $c < d$. Let $\CC ([c,d]_\T,\R^n)$ denote the space of continuous functions defined on $[c,d]_\T$ with values in $\R^n$. Endowed with its usual uniform norm $\Vert \cdot \Vert_\infty$, it is a Banach space. Let $\AC([c,d]_\T,\R^n)$ denote the subspace of absolutely continuous functions. 

Let $t_0 \in [c,d]_\T$ and $q:[c,d]_\T \rightarrow \R^n$. 
It is easily derived from \cite[Theorem 4.1]{caba2} that $q \in \AC([c,d]_\T,\R^n)$ if and only if $q$ is $\DD$-differentiable $\DD$-a.e.\ on $[c,d[_\T$ and satisfies $q^\DD \in \L^1_\T ([c,d[_\T,\R^n)$, and for every $t \in [c,d]_\T$ there holds
$ q(t) = q(t_0) + \int_{[t_0,t[_\T} q^\DD (\tau) \; \DD \tau$
whenever $t \geq t_0$, and
$ q(t) = q(t_0) - \int_{[t,t_0[_\T} q^\DD (\tau) \; \DD \tau $
whenever $ t \leq t_0$.

Assume that $q \in \L^1_\T ([c,d[_\T,\R^n)$, and let $Q$ be the function defined on $[c,d]_\T$ by $Q(t) = \int_{[t_0,t[_\T} q (\tau) \; \DD \tau$ whenever $ t \geq t_0$, and by $ Q(t) = - \int_{[t,t_0[_\T} q (\tau) \; \DD \tau $ whenever $t \leq t_0$.
Then $Q \in \AC([c,d]_\T)$ and $Q^\DD = q$ $\DD$-a.e.\ on $[c,d[_\T$.

Note that, if $q \in \AC([c,d]_\T, \R^n )$ is such that $q^\DD = 0$ $\DD$-a.e.\ on $[c,d[_\T$, then $q$ is constant on $[c,d]_\T$, and that, if $q,q' \in \AC([c,d]_\T, \R^n )$, then $\langle q,q' \rangle_{\R^n} \in \AC ([c,d]_\T, \R )$ and the Leibniz formula~\eqref{eqleibniz} is available $\DD$-a.e.\ on $[c,d[_\T$.

For every $q \in \L^1_\T ([c,d[_\T,\R^n)$, let $\LL_{[c,d[_\T} (q)$ be the set of points $t \in [c,d[_\T$ that are $\DD$-Lebesgue points of $q$. There holds $\mu_\DD (\LL_{[c,d[_\T}(q)) = \mu_\DD ([c,d[_\T) = d-c$, and
\begin{equation}\label{eq9999}
\lim\limits_{\substack{\beta \to 0^+ \\ \beta \in \V^d_s}} \frac{1}{\beta}  \int_{[s,s+\beta[_\T} q(\tau) \, \DD \tau = q(s),
\end{equation}
for every $s \in \LL_{[c,d[_\T}(q) \cap \SS$, where $\V^d_s$ is defined by \eqref{defVbs}.

\begin{remark}\label{footnotelebesguepoints}
Note that the analogous result for $s \in \LL_{[c,d[_\T}(q) \cap \mathrm{LD}$ is not true in general.
Indeed, let $q \in \L^1_\T ([c,d[_\T,\R^n)$ and assume that there exists a point $s \in [c,d[_\T \cap \mathrm{LD} \cap \RR$. Since $\mu_\DD(\{ s \} ) = \mu (s) > 0$, one has $s \in \LL_{[c,d[_\T}(q)$. Nevertheless the limit $\frac{1}{\beta} \int_{[s-\beta,s[_\T} q(\tau) \, \DD \tau $ as $\beta \to 0^+$ with $s-\beta \in \T$, is not necessarily equal to $q(s)$. For instance, consider $\T = [0,1] \cup \{ 2 \}$, $s=1$ and $q$ defined on $\T$ by $q(t)=0$ for every $t \neq 1$ and $q(1)=1$..
\end{remark}

\begin{remark}\label{rem18h05}
Recall that two distinct derivative operators are usually considered in the time scale calculus, namely, the $\DD$-derivative, corresponding to a forward derivative, and the $\nabla$-derivative, corresponding to a backward derivative, and that both of them are associated with a notion of integral. In this article, without loss of generality we consider optimal control problems defined on time scales with a $\DD$-derivative and with a cost function written with the corresponding notion of $\DD$-integral. Our main result, the PMP, is then stated using the notions of right-dense and right-scattered points. All problems and results of our article can be as well stated in terms of $\nabla$-derivative, $\nabla$-integral, left-dense and left-scattered points.
\end{remark}

\subsection{Topological preliminaries}\label{sec_topoprelim}
Let $m \in \N^*$ and let $\Omega$ be a non empty closed subset of $\R^m$. In this section we define the notion of stable $\Omega$-dense direction. In our main result the set $\Omega$ will stand for the set of pointwise constraints on the controls.

\begin{definition}
Let $v \in \Omega$ and $v' \in \R^m$.
\begin{enumerate}
\item We set
$\D (v,v') = \{ 0 \leq \alpha \leq 1\ \vert\ v+\alpha (v'-v) \in \Omega \}$.
Note that $0 \in \D (v,v')$. 
\item We say that $v'$ is a $\Omega$-dense direction from $v$ if $0$ is not isolated in $\D (v,v')$. The set of all $\Omega$-dense directions from $v$ is denoted by $\D (v)$.
\item We say that $v'$ is a stable $\Omega$-dense direction from $v$ if there exists $\varepsilon>0$ such that $v' \in \D (v'')$ for every $v'' \in \B (v,\varepsilon) \cap \Omega$, where $\B (v,\varepsilon)$ is the closed ball of $\R^m$ centered at $v$ and with radius $\varepsilon$.
The set of all stable $\Omega$-dense directions from $v$ is denoted by $\DS (v)$.
\end{enumerate}
\end{definition}

Note that $v' \in \DS (v) $ means that $v'$ is a $\Omega$-dense direction from $v''$ for every $v''\in\Omega$ in a neighbourhood of $v$. In the following, we denote by $\Int$ the interior of a subset. We have the following easy properties.
\begin{enumerate}
\item If $v \in \Int (\Omega)$, then $\DS (v) = \R^m$.
\item If $\Omega=\{ v \}$ then $\DS (v) = \{ v \}$;
\item If $\Omega$ is convex then $\Omega \subset \DS (v)$ for every $v \in \Omega$.
\end{enumerate}

For every $v \in \Omega$, we denote by $\Coad ( \DS (v) )$ the closed convex cone of vertex $v$ spanned by $\DS (v)$, with the agreement that $\Coad ( \DS (v) ) = \{ v \}$ whenever $\DS (v) = \emptyset$. In particular, there holds $v \in \Coad ( \DS (v) )$ for every $v \in \Omega$.

Although elementary, since these notions are new (up to our knowledge), before proceeding with our main result (stated in Section \ref{section2}) we provide the reader with several simple examples illustrating these notions.
Since $\DS (v) = \Coad (\DS (v) ) = \R^m$ for every $v \in \Int (\Omega)$, we focus on elements $v \in \partial\Omega$ in the examples below.

\begin{example}
Assume that $m=1$. The closed convex subsets $\Omega$ of $\R$ having a nonempty interior and such that $\partial\Omega \neq \emptyset$ are closed intervals bounded above or below and not reduced to a singleton. If $\Omega$ is bounded below then $ \DS (\min\Omega) = \Coad (\DS (\min\Omega)) = [\min\Omega,+\infty[$, and if $\Omega$ is bounded above then $\DS (\max\Omega) = \Coad (\DS (\max\Omega)) = ]-\infty,\max\Omega]$.
\end{example}

\begin{example}
Assume that $m=2$ and let $\Omega$ be the convex set of $v=(v_1,v_2)\in\R^2$ such that $v_1\geq 0$, $v_2 \geq 0$ and $v_1^2 + v_2^2 \leq 1$ (see Figure \ref{fig1}).
\begin{figure}[h]
\centerline{\includegraphics[width=6cm]{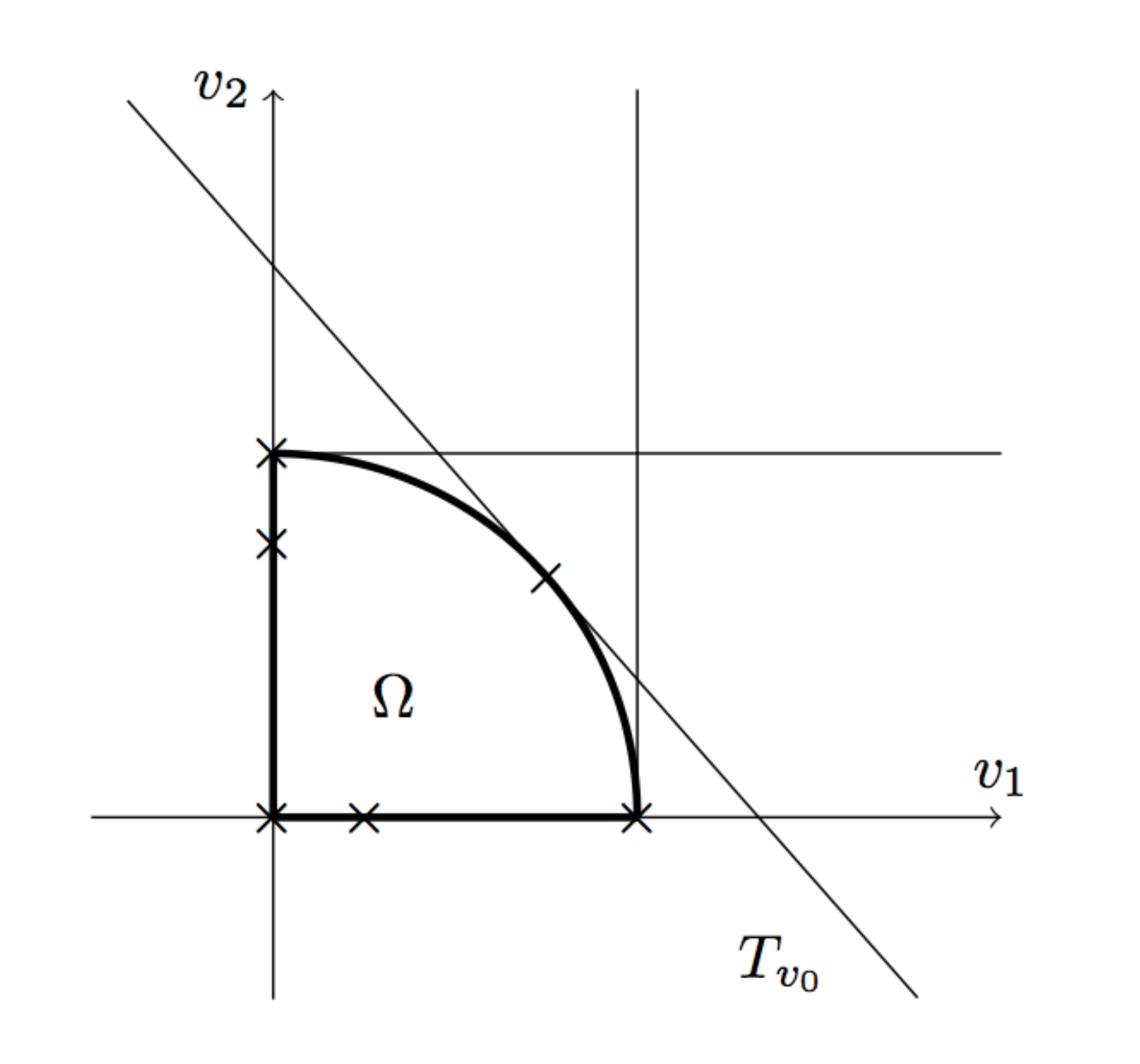}}
\caption{}\label{fig1}
\end{figure}
The stable $\Omega$-dense directions for elements $v \in \partial\Omega$ are given by:
\begin{itemize}
\item if $v=(0,0)$, then $\DS (v) = \Coad (\DS (v) ) = (\R^+)^2$;
\item if $v=(0,v_0)$ with $0 < v_0 < 1$, then $\DS (v) = \Coad (\DS (v) ) = \R^+ \times \R$;
\item if $v=(v_0,0)$ with $0 < v_0 < 1$, then $\DS (v) = \Coad (\DS (v) ) = \R \times \R^+$;
\item if $v=(0,1)$, then $\DS (v) = \{ (v_1,v_2) \in \R^2\ \vert\  v_1 \geq 0,  v_2 < 1 \} \cup \{ v \}$ and $\Coad ( \DS (v) ) = \{ (v_1,v_2) \in \R^2\ \vert\  v_1 \geq 0,  v_2 \leq 1 \}$;
\item if $v=(1,0)$, then $\DS (v) = \{ (v_1,v_2) \in \R^2\ \vert\  v_1 < 1,  v_2 \geq 0 \} \cup \{ v \}$ and $\Coad ( \DS (v) ) = \{ (v_1,v_2) \in \R^2\ \vert\  v_1 \leq 1,  v_2 \geq 0 \}$;
\item if $v=(v_0,\sqrt{1-v_0^2})$ with $0 < v_0 < 1$, then $\DS (v)$ is the union of $\{ v \}$ and of the strict hypograph of $T_{v_0}$, and $\Coad ( \DS (v) )$ is the hypograph of $T_{v_0}$.
\end{itemize}
\end{example}

\begin{remark}
Let $\Omega$ be a non empty closed convex subset of $\R^m$ and let $\PPP$ denote the smallest affine subspace of $\R^m$ containing $\Omega$. For every $v\in\partial\Omega$ that is not a corner point, $\Coad (\DS (v))$ is the half-space of $\PPP$ delimited by the tangent hyperplane (in $\PPP$) of $\Omega$ at $v$, and containing $\Omega$.
\end{remark}

\begin{example}\label{ex1}
Assume that $m=2$ and let $\Omega$ be the set of $v=(v_1,v_2)\in\R^2$ such that $v_2\leq\vert v_1\vert$ (see Figure \ref{fig2}).
\begin{figure}[h]
\centerline{\includegraphics[width=6cm]{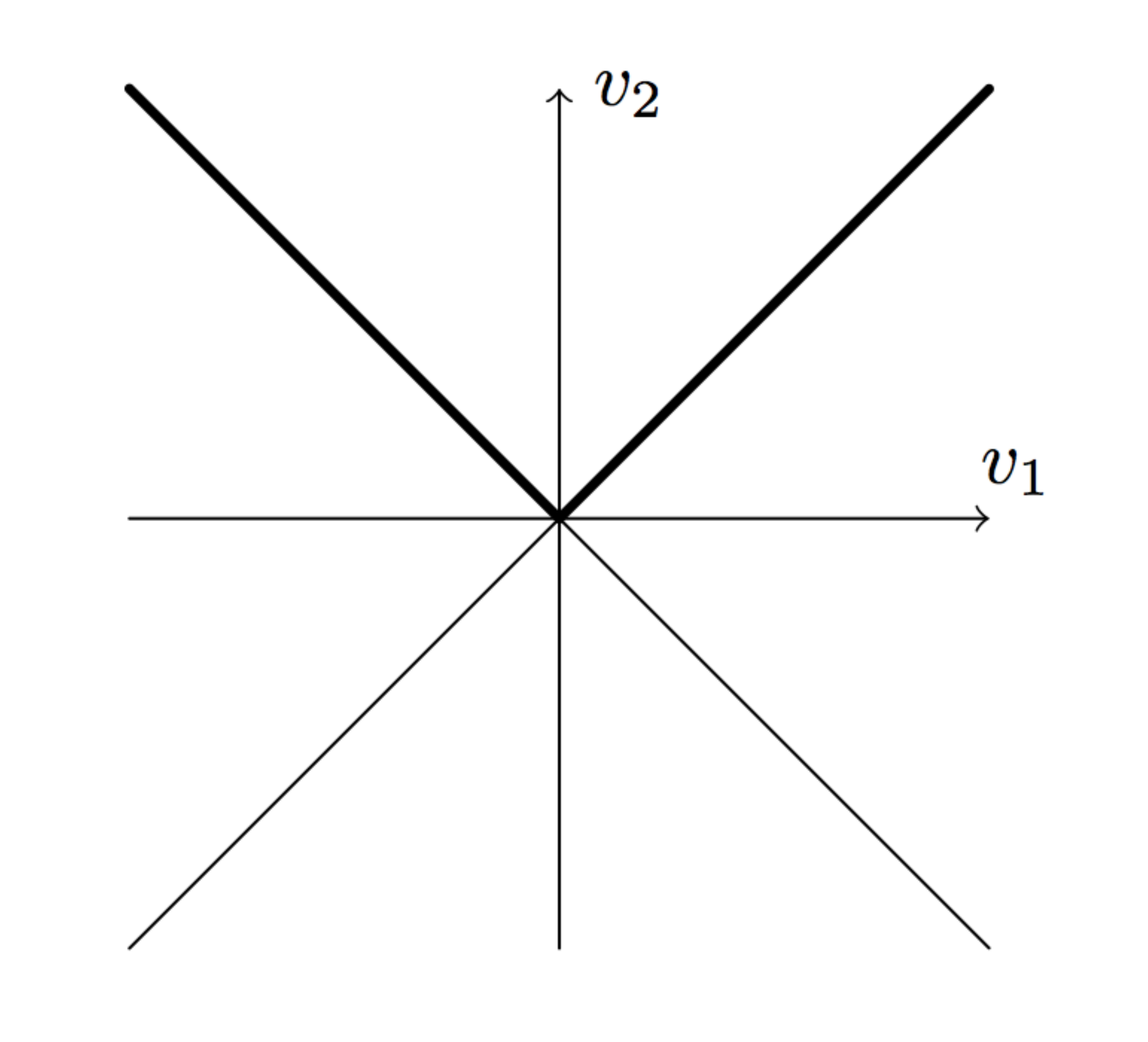}}
\caption{}\label{fig2}
\end{figure}
The stable $\Omega$-dense directions for elements $v \in \partial\Omega$ are given by:
\begin{itemize}
\item if $v=(v_0,\vert v_0 \vert)$ with $v_0 < 0$, then $\D (v) = \DS (v) = 
\{ (v_1,v_2)\in\R^2\ \vert\ v_2\leq -v_1\}$; 
\item if $v=(v_0,\vert v_0 \vert)$ with $v_0 > 0$, then $\D (v) = \DS (v) = \{ (v_1,v_2)\in\R^2\ \vert\ v_2\leq v_1\}$; 
\item if $v=(0,0)$, then $ \D (v) = \Omega$, $\DS (v) = \{ (v_1,v_2)\in\R^2\ \vert\ v_2\leq -\vert v_1\vert \}$; 
\end{itemize}
Note that, in all cases, $\DS(v)$ is a closed convex cone of vertex $v$ and therefore $\Coad ( \DS(v) ) = \DS(v)$.
\end{example}

\begin{example}\label{ex2}
Assume that $m=2$ and let $\Omega$ be the set of $v=(v_1,v_2)\in\R^2$ such that $v_2\leq v_1^2$ (see Figure \ref{fig3}).
Let $v_0\in\R$ and let $T_{v_0}(v_1)=v_0(2v_1-v_0)$ denote the graph of the tangent to $\Omega$ at the point $v = (v_0,v_0^2)$.
\begin{figure}[h]
\centerline{\includegraphics[width=6cm]{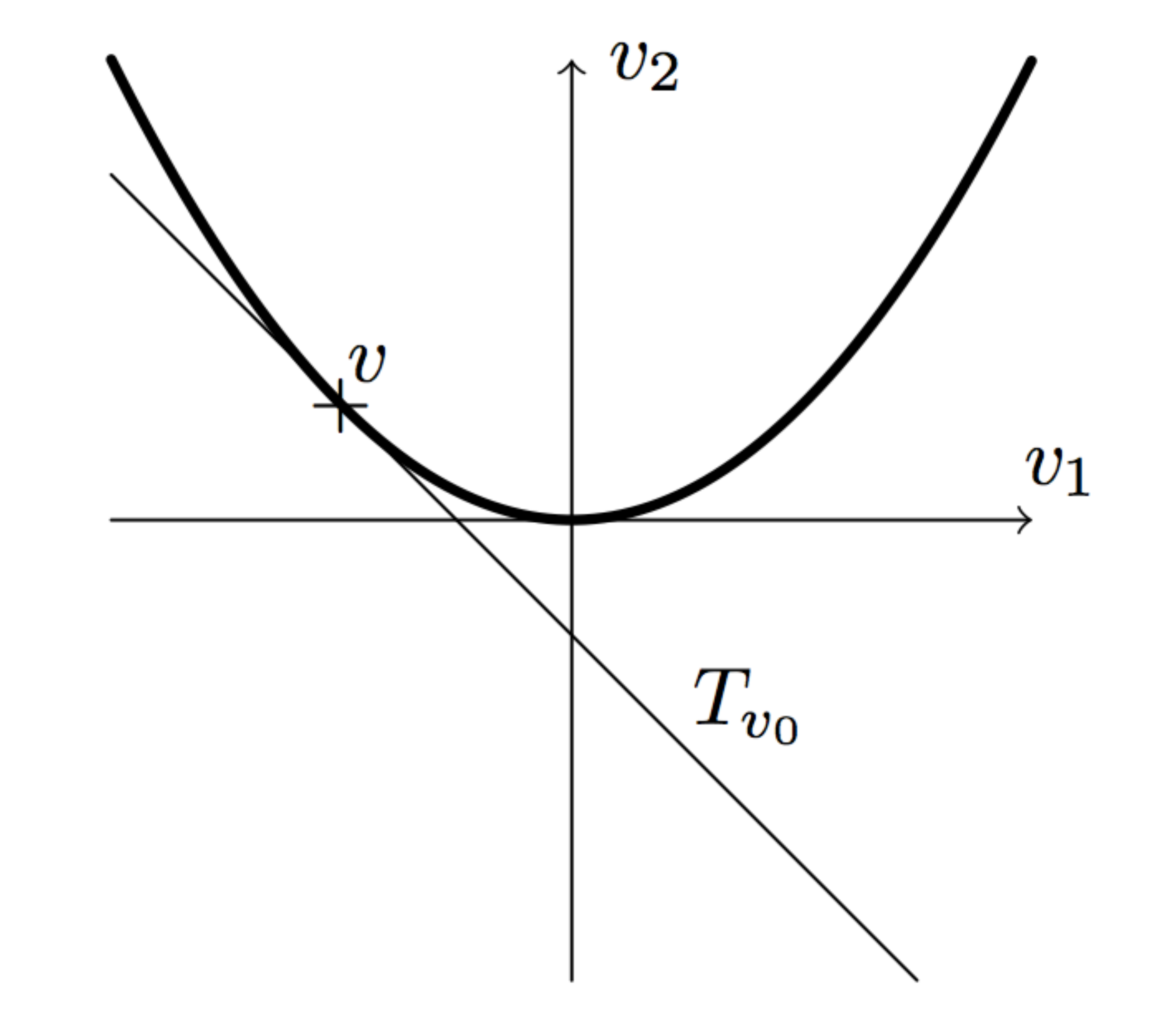}}
\caption{}\label{fig3}
\end{figure}
It is easy to see that $\D (v)$ is the hypograph of $T_{v_0}$, that $\DS (v)$ is the strict hypograph of $T_{v_0}$ (note that $v \notin \DS (v)$), and that $\Coad (\DS (v))$ is the hypograph of $T_{v_0}$. 
\end{example}

\begin{remark}
The above example shows that it may happen that $v \notin \DS (v)$. Actually, it may happen that $\DS (v) = \emptyset$. For example, if $\Omega$ is the unit sphere of $\R^2$, then $\DS (v) = \emptyset$ for every $v \in \Omega$, and hence $\Coad (\DS (v) ) = \{ v \}$.
\end{remark}

\begin{example}\label{ex3}
Assume that $m=2$. 
We set $\Omega =  \cup_{k\in\N} \overline{\Omega}_k \cup \overline{\Omega}_\infty$, where $\Omega_k = \{ (v_1,(1-v_1)/2^k)\ \vert\ 0<v_1<1\}$ for every $k \in \N$, and $\Omega_\infty = \{ (v_1,0)\ \vert\ 0<v_1<1\}$ (see Figure \ref{fig4}). Note that $\Omega$ has an empty interior. Denote by $\overline{v} = (1,0)$.
\begin{figure}[h]
\centerline{\includegraphics[width=9cm]{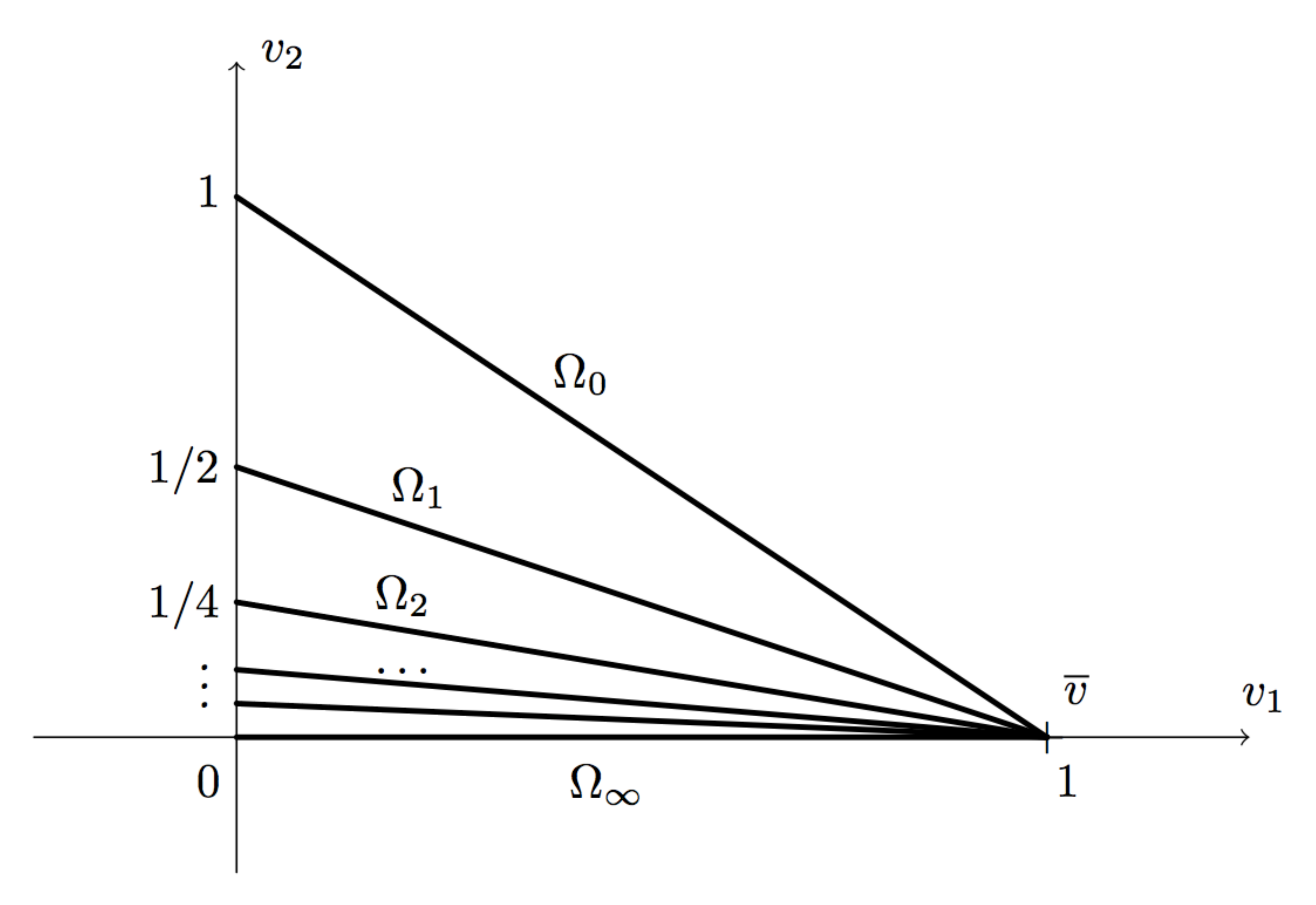}}
\caption{}\label{fig4}
\end{figure}
We have the following properties:
\begin{itemize}
\item if $v \in \Omega_k$ with $k\in\N$, then $\Coad (\DS (v) ) = \DS (v) = \D (v) = \{ (v_1,(1-v_1)/2^k)\ \vert\ v_1\in\R\}$;
\item if $v = (0,1/2^k)$ with $k\in\N$, then $\Coad (\DS (v) ) = \DS (v) = \D (v) = \{ (v_1,(1-v_1)/2^k)\ \vert\ v_1\geq 0\}$;
\item if $v=(v_1,0)$ with $0 < v_1 < 1$, then $\D (v) = \R \times \R^+$ and $\DS (v) = \{ \overline{v} \}$, and thus $\Coad ( \DS (v) ) = [v_1,+\infty[ \times \{ 0 \}$;
\item if $v = (0,0)$, then $\D (v) = (\R^+)^2$ and $\DS (v) = \{ \overline{v} \}$, and thus $\Coad ( \DS (v) ) = \R^+ \times \{ 0 \}$;
\item if $v = \overline{v}$, then 
$\D(\overline{v}) = \cup_{k\in\N} \{ (v_1,(1-v_1)/2^k)\ \vert\ v_1\leq 1\} \cup \{ (v_1,0)\ \vert\ v_1\leq 1\} $ and $\Coad (\DS (v) ) = \DS (\overline{v}) = \{ \overline{v} \}$.
\end{itemize}
\end{example}

\subsection{General nonlinear optimal control problem on time scales}\label{section2}
Let $n$ and $m$ be nonzero integers, and let $\Omega$ be a non empty closed subset of $\R^{m}$.
Throughout the article, we consider the general nonlinear control system on the time scale $\T$
\begin{equation} \label{DD-CS}
q^\DD(t) = f (q(t),u(t),t),
\end{equation}
where $f:\R^{n} \times \R^{m} \times \T \rightarrow \R^{n}$ is a continuous function of class $\CC^1$ with respect to its two first variables, and where the control functions $u$ belong to $\L^\infty_{\mathrm{loc},\T} (\TS;\Omega)$. 

Before defining an optimal control problem associated with the control system \eqref{DD-CS}, the first question that has to be addressed is the question of the existence and uniqueness of a solution of \eqref{DD-CS}, for a given control function and a given initial condition $q(a)=q_a\in\R^n$.
Since there did not exist up to now in the existing literature any Cauchy-Lipschitz like theorem, sufficiently general to cover such a situation, in the companion paper \cite{bour10} we derived a general Cauchy-Lipschitz (or Picard-Lindel\"of) theorem for general nonlinear systems posed on time scales, providing existence and uniqueness of the maximal solution of a given $\DD$-Cauchy problem under suitable assumptions like regressivity and local Lipschitz continuity, and discussed some related issues like the behavior of maximal solutions at terminal points.

Setting $\U= \L^\infty_{\mathrm{loc},\T} (\TS;\R^m)$, let us first recall the notion of a solution of \eqref{DD-CS}, for a given control $u\in\U$ (see \cite[Definitions 6 and 7]{bour10}).
The couple $(q,I_\T)$ is said to be a solution of \eqref{DD-CS} if $I_\T$ is an interval of $\T$ satisfying $a \in I_\T$ and $I_\T \backslash \{ a \} \neq \emptyset$, if $q \in \AC ([a,b]_\T,\R^n)$ and \eqref{DD-CS} holds for $\DD$-a.e.\ $t \in [a,b[_\T$, for every $b \in I_\T \backslash \{ a \}$.

According to \cite[Theorem 1]{bour10}, for every control $u\in\U$ and every $q_a \in \R^n$, there exists a unique maximal solution $q(\cdot,u,q_a)$ of \eqref{DD-CS}, such that $q(a)=q_a$, defined on the maximal interval $I_\T(u,q_a)$. The word \textit{maximal} means that $q(\cdot,u,q_a)$ is an extension of any other solution. 
Note that
$q(t,u,q_a) = q_a + \int_{[a,t[_\T} f(q(\tau,u,q_a),u(\tau),\tau) \, \DD \tau,$
for every $t \in I_\T (u,q_a)$ (see \cite[Lemma~1]{bour10}), and that
either $I_\T (u,q_a) = \T$, that is, $q(\cdot,u,q_a)$ is a \textit{global} solution of \eqref{DD-CS}, or $I_\T (u,q_a) = [a,b[_\T$ where $b \in \T \backslash \{a\}$ is a left-dense point of $\T$, and in this case, $q(\cdot,u,q_a)$ is not bounded on $I_\T (u,q_a)$
(see \cite[Theorem~2]{bour10}).

These results are instrumental to define the concept of admissible control.

\begin{definition}\label{defadm}
For every $q_a\in\R^n$, the control $u\in\U$ is said to be \textit{admissible} on $[a,b[_\T$ for some given $b \in \T \backslash \{a\}$ whenever $q(\cdot,u,q_a)$ is well defined on $[a,b]_\T$, that is, $b \in I_\T (u,q_a)$.
\end{definition}

We are now in a position to define rigorously a general optimal control problem on the time scale $\T$.

Let $j \in \N^*$ and $\S$ be a non empty closed convex subset of $\R^j$. Let $f^0:\R^{n} \times \R^{m} \times \T \rightarrow\R$ be a continuous function of class $\CC^1$ with respect to its two first variables, and $g:\R^n\times\R^n\rightarrow\R^j$ be a function of class $\CC^1$.
In what follows the subset $\S$ and the function $g$ account for constraints on the initial and final conditions of the control problem.

Throughout the article, we consider the optimal control problem on $\T$, denoted in short $\OCP$, of determining a trajectory $q^*(\cdot)$ defined on $[a,b^*]_\T$, solution of \eqref{DD-CS} and associated with a control $u^*\in \L^\infty_{\T}([a,b^*[_\T;\Omega)$, minimizing the cost function
\begin{equation}\label{cost}
C(b,u) = \int_{[a,b[_\T} f^0 (q (\tau), u(\tau), \tau ) \, \DD \tau
\end{equation}
over all possible trajectories $q(\cdot)$ defined on $[a,b]_\T$, solutions of \eqref{DD-CS} and associated with an admissible control $u\in \L^\infty_{\T}([a,b[_\T;\Omega)$, with $b\in\T \backslash \{a\}$, and satisfying $g(q(a),q(b))\in \S$.

The final time can be fixed or not. If it is fixed then $b^*=b$ in $\OCP$.

\subsection{Pontryagin Maximum Principle}\label{section2bis1}
In the statement below, the \textit{orthogonal} of $\S$ at a point $x\in \S$ is defined by
\begin{equation}\label{defortho}
\O_\S (x) = \{ x' \in \R^j\ \vert \ \forall x'' \in \S, \; \langle x', x''-x \rangle_{\R^j} \leq 0 \}. 
\end{equation}
It is a closed convex cone containing $0$.

The \textit{Hamiltonian} of the optimal control problem $\OCP$ is the function $H:\R^n \times \R^m \times \R^n \times \R \times \T \rightarrow \R$ defined by
$H(q,u,p,p^0,t)=\langle p, f(q,u,t) \rangle_{\R^n} + p^0 f^0(q,u,t).$

\begin{theorem}[Pontryagin Maximum Principle]\label{thmmain}
Let $b^* \in \T \backslash \{a\}$. If the trajectory $q^*(\cdot)$, defined on $[a,b^*]_\T$ and associated with a control $u^*\in \L^\infty_{\T}([a,b^*[_\T;\Omega)$, is a solution of $\OCP$,
then there exist $p^0 \leq 0$ and $\psi \in \R^j$, with $(p^0,\psi)\neq (0,0)$, and there exists a mapping $p(\cdot) \in \AC ([a,b^*]_\T,\R^n)$ (called adjoint vector), such that there holds
\begin{equation}\label{extremal_equations}
q^{*\DD}(t) = \frac{\partial H}{\partial p} (q^*(t),u^*(t),p^\sigma(t),p^0,t), \qquad
p^\DD(t) = - \frac{\partial H}{\partial q} (q^*(t),u^*(t),p^\sigma(t),p^0,t) ,
\end{equation}
for $\DD$-a.e.\ $t\in[a,b^*[_\T$.
Moreover, there holds
\begin{equation}\label{eqrscase}
\Big\langle \frac{\partial H}{\partial u} (q^*(r),u^*(r),p^\sigma(r),p^0,r) , v-u^*(r) \Big\rangle_{\R^m} \leq 0 , 
\end{equation}
for every $r \in [a,b^*[_\T \cap \RR$ and every $v \in \Coad ( \DS (u^*(r)) )$, and
\begin{equation}\label{eqrdcase}
H(q^*(s),u^*(s),p^\sigma(s),p^0,s) = \max_{v\in\Omega} H(q^*(s),v,p^\sigma(s),p^0,s) ,
\end{equation}
for $\DD$-a.e.\ $s \in [a,b^*[_\T \cap \SS$.

Besides, one has the transversality conditions on the initial and final adjoint vector
\begin{equation}\label{transv_cond}
p(a) = - \Big( \frac{\partial g}{\partial q_1} (q^*(a),q^*(b^*)) \Big)^{\!\mathrm{T}}  \psi,\qquad
p(b^*) =  \Big( \frac{\partial g}{\partial q_2} (q^*(a),q^*(b^*)) \Big)^{\!\mathrm{T}}  \psi,
\end{equation}
and $-\psi \in \O_\S (g (q^*(a),q^*(b^*)))$.

Furthermore, if the final time $b^*$ is not fixed in $\OCP$, and if additionally $b^*$ belongs to the interior of $\T$ for the topology of $\R$, then
\begin{equation}\label{transv_cond_time}
\max_{v\in\Omega} H(q^*(b^*),v,p^\sigma(b^*),p^0,b^*) = 0,
\end{equation}
and if $H$ is moreover autonomous (that is, does not depend on $t$), then
\begin{equation}\label{eqnullinthamilt}
\int_{[a,b^*[_\T} H(q^*(t),u^*(t),p^\sigma(t),p^0) \, \DD t = 0.
\end{equation}
\end{theorem}

Theorem \ref{thmmain} is proved in Section \ref{part3}.

\begin{remark}[PMP for optimal control problems with parameters]\label{remPMPparam}
Before proceeding with a series of remarks and comments, we provide a version of the PMP for optimal control problems with parameters.
Let $\Lambda$ be a Banach space. 
We consider the general nonlinear control system with parameters on the time scale $\T$
\begin{equation} \label{DD-CS-lambda}
q^\DD(t) = f (\lambda,q(t),u(t),t),
\end{equation}
where $f: \Lambda \times \R^{n} \times \R^{m} \times \T \rightarrow \R^{n}$ is a continuous function of class $\CC^1$ with respect to its three first variables, and where $u\in\U$ as before. The notion of admissibility is defined as before.
Let $f^0: \Lambda \times \R^n \times \R^m \times \T \rightarrow\R$ be a continuous function of class $\CC^1$ with respect to its three first variables, and $g:\Lambda\times\R^n\times\R^n\rightarrow\R^j$ be a function of class $\CC^1$.

We consider the optimal control problem on $\T$, denoted in short $\OCP^\lambda$, of determining a trajectory $q^*(\cdot)$ defined on $[a,b^*]_\T$, solution of \eqref{DD-CS-lambda} and associated with a control $u^*\in \L^\infty_{\T}([a,b^*[_\T;\Omega)$ and with a parameter $\lambda^*\in\Lambda$, minimizing the cost function
$C(\lambda,b,u) = \int_{[a,b[_\T} f^0 (\lambda, q (\tau), u(\tau), \tau ) \, \DD \tau$
over all possible trajectories $q(\cdot)$ defined on $[a,b]_\T$, solutions of \eqref{DD-CS-lambda} and associated with $\lambda\in\Lambda$ and with an admissible control $u\in  \L^\infty_{\T}([a,b[_\T;\Omega)$, with $b\in\T \backslash \{a\}$, and satisfying $g(\lambda,q(a),q(b))\in \S$. The final time can be fixed or not.

The \textit{Hamiltonian} of $\OCP^\lambda$ is the function $H: \Lambda \times \R^n \times \R^m \times \R^n \times \R \times \T \rightarrow \R$ defined by
$$
H(\lambda,q,u,p,p^0,t) = \langle p, f(\lambda,q,u,t) \rangle_{\R^n} + p^0 f^0(\lambda,q,u,t).
$$

If the trajectory $q^*(\cdot)$, defined on $[a,b^*]_\T$ and associated with a control $u^*\in  \L^\infty_{\T}([a,b^*[_\T;\Omega)$ and with a parameter $\lambda^*\in\Lambda$, is a solution of $\OCP^\lambda$, then all conclusions of Theorem \ref{thmmain} (except \eqref{eqnullinthamilt}) hold, and moreover
\begin{equation}\label{condHlambda}
\int_{[a,b^*[_\T} \frac{\partial H}{\partial \lambda} (\lambda^*,q^*(t),u^*(t),p^\sigma(t),p^0,t) \, \DD t + \Big\langle \frac{\partial g}{\partial \lambda} (\lambda^*,q^*(a),q^*(b^*)) , \psi \Big\rangle_{\R^j} = 0.
\end{equation}
This additional statement is proved as well in Section \ref{part3}.
\end{remark}

\begin{remark}
As is well known, the Lagrange multiplier $(p^0,\psi)$ (and thus the triple $(p(\cdot),p^0,\psi)$) is defined up to a multiplicative scalar. Defining as usual an \textit{extremal} as a quadruple $(q(\cdot),u(\cdot),p(\cdot),p^0)$ solution of the above equations, an extremal is said to be \textit{normal} whenever $p^0\neq 0$ and \textit{abnormal} whenever $p^0=0$. The component $p^0$ corresponds to the Lagrange multiplier associated with the cost function. In the normal case $p^0\neq 0$ it is usual to normalize the Lagrange multiplier so that $p^0=-1$.
Finally, note that the convention $p^0\leq 0$ in the PMP leads to a maximization condition of the Hamiltonian (the convention $p^0\geq 0$ would lead to a minimization condition).
\end{remark}

\begin{remark}
As already mentioned in Remark \ref{rem18h05}, without loss of generality we consider in this article optimal control problems defined with the notion of $\DD$-derivative and $\DD$-integral. These notions are naturally associated with the concepts of right-dense and right-scattered points in the basic properties of calculus (see Section~\ref{section1}).
Therefore, when using a $\DD$-derivative in the definition of $\OCP$ one cannot hope to derive in general, for instance, a maximization condition at left-dense points
(see the counterexample of Remark \ref{footnotelebesguepoints}).
\end{remark}

\begin{remark}
In the classical continuous-time setting, it is well known that the maximized Hamiltonian along the optimal extremal, that is, the function $t\mapsto  \max_{v \in \Omega} H(q^*(t),v,p^\sigma(t),p^0,t)$, is Lipschitzian on $[a,b^*]$, and if the dynamics are autonomous (that is, if $H$ does not depend on $t$) then this function is constant.
Moreover, if the final time is free then the maximized Hamiltonian vanishes at the final time.

In the discrete-time setting and a fortiori in the general time scale setting, none of these properties do hold any more in general (see Examples \ref{ex00} and \ref{ex000} below).
The non constant feature is due in particular to the fact that the usual formula of derivative of a composition does not hold in general time scale calculus.
\end{remark}

\begin{remark}
The PMP is derived here in a general framework. We do not make any particular assumption on the time scale $\T$, and do not assume that the set of control constraints $\Omega$ is convex or compact.
In Section~\ref{section10}, we discuss the strategy of proof of Theorem~\ref{thmmain} and we explain how the generality of the framework led us to choose a method based on a variational principle rather than one based on a fixed-point theorem.

We do not make any convexity assumption on the dynamics $(f,f^0)$. As a consequence, and as is well known in the discrete case (see e.g.\ \cite[p.\ 50--63]{bolt}), at right-scattered points the maximization condition \eqref{eqrdcase} does not hold true in general and must be weakened into \eqref{eqrscase} (see Remark \ref{rem8h22}).
\end{remark}

\begin{remark}\label{remcoad}
The inequality~\eqref{eqrscase}, valuable at right-scattered points, can be written as
$$\frac{\partial H}{\partial u} (q^*(r),u^*(r),p^\sigma(r),p^0,r) \in \O_{\Coad (\DS(u^*(r)))} (u^*(r)).$$

In particular, if $u^*(r) \in \Int (\Omega)$ then
$\frac{\partial H}{\partial u} (q^*(r),u^*(r),p^\sigma(r),p^0,r) = 0$. This equality holds true at every right-scattered point if for instance $\Omega = \R^m$ (and also at right-dense points: this is the context of what is usually referred to as the \textit{weak PMP}, see \cite{hils2} where this weaker result is derived on general time scales for shifted control systems).

If $\Omega$ is convex, since $u^*(r) \in \Omega \subset \Coad (\DS(u^*(r)))$, then there holds in particular
$$\frac{\partial H}{\partial u} (q^*(r),u^*(r),p^\sigma(r),p^0,r) \in \O_{\Omega} (u^*(r)),$$
for every $r \in [a,b[_\T \cap \RR$.

Note that, if the inequality \eqref{eqrscase} is strict then $u^*(r)$ satisfies a local maximization condition on $\Coad ( \DS (u^*(r)) )$ (see also \cite[p.\ 74--75]{bolt}).
\end{remark}

\begin{remark}\label{rem8h22}
In the classical continuous-time case, all points are right-dense and consequently, Theorem \ref{thmmain} generalizes the usual continuous-time PMP where the maximization condition \eqref{eqrdcase} is valid $\mu_L$-almost everywhere (see \cite[Theorem 6 p.\ 67]{pont}).

In the discrete-time setting, the possible failure of the maximization condition is a well known fact (see e.g.\ \cite[p.\ 50--63]{bolt}), and a fortiori in the time scale setting the maximization condition cannot be expected to hold in general at right-scattered points (see counterexamples below).

Many works have been devoted to derive a PMP in the discrete-time setting (see e.g.\ \cite{blot,bolt,cano,Halkin,holt2,holt,mord}). Since the maximization condition cannot be expected to hold true in general for discrete-time optimal control problems, it must be replaced with a weaker condition, of the kind \eqref{eqrscase}, involving the derivative of $H$ with respect to $u$. Such a kind of inequality is provided in \cite[Theorem~42.1 p.\ 330]{bolt} for finite horizon problems and in \cite{blot} for infinite horizon problems. Our condition \eqref{eqrscase} is of a more general nature, as discussed next.
In \cite{holt2,holt,seth} the authors assume \textit{directional convexity}, that is, for all $(v,v') \in \Omega^2$ and every $\theta \in [0,1]$, there exists $v_\theta \in \Omega$ such that
$$ f (q,v_\theta,t) = \theta f (q,v,t) + (1-\theta) f (q,v',t), \quad f^0 (q,v_\theta,t) \leq \theta f^0 (q,v,t) + (1-\theta) f^0 (q,v',t), $$
for every $q \in \R^n$ and every $t \in \T$; and under this assumption they derive the maximization condition in the discrete-time case (see also \cite{cano} and \cite[p.\ 235]{seth}). 
Note that this assumption is satisfied whenever $\Omega$ is convex, the dynamics $f$ is affine with respect to $u$, and $f^0$ is convex in $u$ (which implies that $H$ is concave in $u$).
We refer also to \cite{mord} where it is shown that, in the absence of such convexity assumptions, an approximate maximization condition can however be derived.

Note that, under additional assumptions, \eqref{eqrscase} implies the maximization condition. More precisely, let $r \in [a,b[_\T \cap \RR$ and let $(q^*(\cdot),u^*(\cdot),p(\cdot),p^0)$ be the optimal extremal of Theorem \ref{thmmain}. Let $r \in [a,b[_\T \cap \RR$.
If the function $u\mapsto H(q^*(r), u , p^\sigma(r),p^0,r )$ is concave on $\R^m$, then the inequality \eqref{eqrscase} implies that
$$
H(q^*(r), u^*(r) , p^\sigma(r),p^0,r ) = \max_{v \in \Coad (\DS(u^*(r)))} H(q^*(r), v , p^\sigma(r),p^0,r ) .
$$
If moreover $\Omega \subset \Coad (\DS(u^*(r)))$ (this is the case if $\Omega$ is convex), since $u^* (r) \in \Omega$, it follows that
$$
H(q^*(r), u^*(r) , p^\sigma(r),p^0,r ) = \max_{v \in \Omega} H(q^*(r), v , p^\sigma(r),p^0,r ) .
$$
Therefore, in particular, if $H$ is concave in $u$ and $\Omega$ is convex then the maximization condition holds as well at every right-scattered point.
\end{remark}

\begin{remark}\label{remihp}
It is interesting to note that, if $H$ is convex in $u$ then a certain minimization condition can be derived at every right-scattered point, as follows.

For every $v \in \Omega$, let $\Opp (v) = \{ 2v-v'\ \vert\ v' \in \Coad (\DS (v)) \}$ denote the symmetric of $\Coad (\DS (v))$ with respect to the center $v$.
It obviously follows from \eqref{eqrscase} that
\begin{equation}\label{eqrscase2}
\Big\langle \frac{\partial H}{\partial u} (q^*(r),u^*(r),p^\sigma(r),p^0,r) , v-u^*(r) \Big\rangle_{\R^m} \geq 0,  
\end{equation}
for every $r \in [a,b[_\T \cap \RR$ and every $v \in \Opp(u^*(r))$.
If $H$ is convex in $u$ on $\R^m$, then the inequality \eqref{eqrscase2} implies that
\begin{equation}\label{hammin}
H(q^*(r), u^*(r) , p^\sigma(r),p^0,r ) = \min_{v \in \Opp (u^*(r))} H(q^*(r), v , p^\sigma(r),p^0,r ) 
\end{equation}
\end{remark}

We next provide several very simple examples illustrating the previous remarks.

\begin{example}\label{ex00}
Here we give a counterexample showing that, although the final time is not fixed, the maximized Hamiltonian may not vanish.

Set $\T = \N$, $n=m=1$, $f (q,u,t) = u$, $ f^0(q,u,t) = 1$, $\Omega = [0,1]$, $j=2$, $g(q_1,q_2) = (q_1,q_2)$ and $\S = \{ 0 \} \times \{ 3/2 \}$. The corresponding optimal control problem is the problem of steering the discrete-time control one-dimensional system $q(n+1)=q(n)+u(n)$ from $q(0)=0$ to $q(b)=3/2$ in minimal time, with control constraints $0\leq u(k)\leq 1$.
It is clear that the minimal time is $b^*=2$, and that any control $u$ such that $0\leq u(0)\leq 1$, $0\leq u(1)\leq 1$, and $u(0)+u(1) = 3/2$, is optimal.

Among these optimal controls, consider $u^*$ defined by $u^*(0) = 1/2$ and $u^*(1) = 1$. 
Consider $\psi$, $p^0\leq 0$ and $p(\cdot)$ the adjoint vector whose existence is asserted by the PMP.
Since $u^*(0) \in \Int (\Omega)$, it follows from \eqref{eqrscase} that $p(1) = 0$.
The Hamiltonian is $H(q,u,p,p^0,t) = p u + p^0$, and since it is independent of $q$, it follows that $p(\cdot)$ is constant and thus equal to $0$. In particular, $p(0) = p(2) = 0$ and hence $\psi = 0$. From the nontriviality condition $(p^0,\psi)\neq (0,0)$ we infer that $p^0\neq 0$. Therefore the maximized Hamiltonian at the final time is here equal to $p^0$ and thus is not equal to $0$.
\end{example}

\begin{example}\label{ex0}
Here we give a counterexample (in the spirit of \cite[Examples 10.1-10.4 p.\ 59--62]{bolt}) showing the failure of the maximization condition at right-scattered points. 

Set $\T = \{ 0,1,2 \}$, $n=m=1$, $ f (q,u,t) = u-q$, $f^0(q,u,t) = 2q^2-u^2$, $\Omega = [0,1]$, $j=1$, $g(q_1,q_2) = q_1$ and $\S = \{ 0 \}$. 
Any solution of the resulting control system is such that $q(0)=0$, $q(1)=u(0)$, $q(2)=u(1)$, and its cost is equal to $u(0)^2 - u(1)^2$. It follows that the optimal control $u^*$ is unique and is such that $u^*(0) = 0$ and $u^*(1) = 1$. 
The Hamiltonian is $H(q,u,p,p^0,t) = p (u-q) + p^0 (2q^2 -u^2)$.
Consider $\psi$, $p^0\leq 0$ and $p(\cdot)$ the adjoint vector whose existence is asserted by the PMP.
Since $g$ does not depend on $q_2$, it follows that $p(2)=0$, and from the extremal equations we infer that $p(1)=0$ and $p(0)=0$. Therefore $\psi = 0$ and hence $p^0 \neq 0$ (nontriviality condition) and we can assume that $p^0=-1$.
It follows that the maximized Hamiltonian is equal to $-p^0=1$ at $r=0,1,2$, whereas
$H(q^*(0),u^*(0),p(1),p^0,0)= 0$.
In particular, the maximization condition \eqref{eqrdcase} is not satisfied at $r = 0 \in \RR$ (note that it is however satisfied at $r=1$).

Note that, in accordance with the fact that $H$ is convex in $u$ and $\Opp(u^*(0)) = ]-\infty,0]$ and $\Opp(u^*(1)) = [1,+\infty[$, the minimization condition \eqref{hammin} is indeed satisfied (see Remark \ref{remihp}).
\end{example}

\begin{example}\label{ex000}
Here we give a counterexample in which, although the Hamiltonian is autonomous (independent of $t$), the maximized Hamiltonian is not constant over $\T$.

Set $\T = \{ 0,1,2 \}$, $n=m=1$, $f(q,u,t)=u-q$, $f^0(q,u,t) = (u^2-q^2)/2$, $j=1$, $g(q_1,q_2) = q_1$, $\S = \{ 1 \}$, $\Omega = [0,1]$ and $b=2$. 
Any solution of the resulting control system is such that $q(0)=1$, $q(1)=u(0)$, $q(2)=u(1)$, and its cost is equal to $(u(1)^2 - 1)/ 2$. It follows that any control $u$ such that $u(1) = 0$ is optimal (the value of $u(0)$ is arbitrary). Consider the optimal control $u^*$ defined by $u^*(0) = u^*(1) = 0$, and let $q^*(\cdot)$ be the corresponding trajectory. Then $q^*(0)=1$ and $q^*(1)=q^*(2)=0$.
The Hamiltonian is $H(q,u,p,p^0,t) = p (u-q) + p^0 (u^2 -q^2)/2$.
Consider $\psi$, $p^0\leq 0$ and $p(\cdot)$ the adjoint vector whose existence is asserted by the PMP.
Since $g$ does not depend on $q_2$, it follows that $p(2)=0$, and from the extremal equations we infer that $p(1)=0$ and $p(0)=-p^0$. In particular, from the nontriviality condition one has $p^0\neq 0$ and we can assume that $p^0=-1$.
Therefore $H(q^*(0),v,p(1),p^0,0)= 1/2-v^2 $ and $H(q^*(1),v,p(2),p^0,1)= - v^2 /2$, and it easily follows that that the maximization condition holds at $r=0$ and $r=1$. This is in accordance with the fact that $H$ is concave in $u$ and $\Omega$ is convex.
Moreover, the maximized Hamiltonian is equal to $1/2$ at $r=0$, and to $0$ at $r=1$ and $r=2$.
\end{example}

\section{Proof of the main result}\label{part3}
\subsection{Preliminary comments}\label{section10}
There exist several proofs of the continuous-time PMP in the literature. Mainly they can be classified as variants of two different approaches: the first of which consists of using a fixed point argument, and the second consists of using Ekeland's Variational Principle.

More precisely, the classical (and historical) proof of \cite{pont} relies on the use of the so-called needle-like variations combined with a fixed point Brouwer argument (see also \cite{hest,lee}). There exist variants, relying on the use of a conic version of the Implicit Function Theorem (see \cite{agrach} or \cite{HaberkornTrelat,trel3}), the proof of which being however based on a fixed point argument. The proof of \cite{bres} uses a separation theorem (Hahn-Banach arguments) for cones combined with the Brouwer fixed point theorem. We could cite many other variants, all of them relying, at some step, on a fixed point argument.

The proof of \cite{ekel} is of a different nature and follows from the combination of needle-like variations with Ekeland's Variational Principle. It does not rely on a fixed point argument. By the way note that this proof leads as well to an approximate PMP (see \cite{ekel}), and withstands generalizations to the infinite dimensional setting (see e.g.\ \cite{liyo})

Note that, in all cases, needle-like variations are used to generate the so-called Pontryagin cone, serving as a first-order convex approximation of the reachable set.
The adjoint vector is then constructed by propagating backward in time a Lagrange multiplier which is normal to this cone.
Roughly, needle-like variations are kinds of perturbations of the reference control in $\L^1$ topology (perturbations with arbitrary values, over small intervals of time) which generate perturbations of the trajectories in $\CC^0$ topology.

Due to obvious topological obstructions, it is evident that the classical strategy of needle-like variations combined with a fixed point argument cannot hold in general in the time scale setting. At least one should distinguish between dense points and scattered points of $\T$. But even this distinction is not sufficient.
Indeed, when applying the Brouwer fixed point Theorem to the mapping built on needle-like variations (see \cite{lee,pont}), it appears to be crucial that the domain of this mapping be convex. Roughly speaking, this domain consists of the product of the intervals of the spikes (intervals of perturbation). This requirement obviously excludes the scattered points of a time scale (which have anyway to be treated in another way), but even at some right-dense point $s \in \SS$, there does not necessarily exist $\varepsilon>0$ such that $[s,s+\varepsilon]\subset\T$.
At such a point we can only ensure that $0$ is not isolated in the set $\{ \beta \geq 0\ \vert\ s+\beta \in \T \}$. In our opinion this basic obstruction makes impossible the use of a fixed point argument in order to derive the PMP on a general time scale.
Of course to overcome this difficulty one can assume that the $\mu_\DD$-measure of right-dense points not admitting a right interval included in $\T$ is zero. This assumption is however not very natural and would rule out time scales such as a generalized Cantor set having a positive $\mu_L$-measure.
Another serious difficulty that we are faced with on a general time scale is the technical fact that the formula \eqref{eq9999}, accounting for Lebesgue points, is valid only for $\beta$ such that $s+ \beta \in \T$. Actually if $s+ \beta \notin \T$ then \eqref{eq9999} is not true any more in general (it is very easy to construct a time scale $\T$ for which \eqref{eq9999} fails whenever $s+ \beta \notin \T$, even with $q=1$). Note that the concept of Lebesgue point is instrumental in the classical proof of the PMP in order to ensure that the needle-like variations can be built at different times\footnote{More precisely, what is used in the approximate continuity property (see e.g.\ \cite{evan}).} (see \cite{lee,pont}). 
On a general time scale this technical point would raise a serious issue\footnote{We are actually able to overcome this difficulty by considering multiple variations at right-scattered points, however this requires to assume that the set $\Omega$ is locally convex. The proof that we present further does not require such an assumption.}.

The proof of the PMP that we provide in this article is based on Ekeland's Variational Principle, which permits to avoid the above obstructions and happens to be well adapted for the proof of a general PMP on time scales. It requires however the treatment of other kinds of technicalities, one of them being the concept of stable $\Omega$-dense direction that we were led to introduce. Another point is that Ekeland's Variational Principle requires a complete metric space, which has led us to assume that $\Omega$ is closed (see Footnote \ref{footnote_Omegaclosed}).

\begin{remark}\label{remZhan}
Recall that a weak PMP (see Remark \ref{remcoad}) on time scales is proved in \cite{hils2} for shifted optimal control problems (see also \cite{hils1}). A similar result can be derived in an analogous way for the non shifted optimal control problems \eqref{DD-CS} considered here.
Since then, deriving the (strong) PMP on time scales was an open problem. 
While we were working on the contents of the present article (together with the companion paper \cite{bour10}), at some step we discovered the publication of the article \cite{zhan2}, in which the authors claim to have obtained a general version of the PMP. As in our work, their approach is based on Ekeland's Variational Principle.
However, as already mentioned in the introduction, many arguments thereof are erroneous, and we believe that their errors cannot be corrected easily.

Although it is not our aim to be involved in controversy, since we were incidentally working in parallel on the same subject (deriving the PMP on time scales), we provide hereafter some evidence of the serious mistakes contained in \cite{zhan2}.

Note that, in order to derive a maximization condition $\DD$-almost everywhere (even at right-scattered points), as in \cite{holt2,holt,seth} the authors of \cite{zhan2} assume directional convexity of the dynamics (see Remark \ref{rem8h22} for the definition).

A first serious mistake in \cite{zhan2} is the fact that, in the application of Ekeland's Variational Principle, the authors use two different distances, depending on the nature of the point of $\T$ under consideration (right-scattered or dense).
As is usual in the proof of the PMP by Ekeland's Principle, the authors deduce from considerations on sequences of perturbation controls $u^\varepsilon$ the existence of Lagrange multipliers $\varphi_0$ and $\psi_0$; the problem is that these multipliers are built separately for right-scattered and dense points (see \cite[(32), (43), (52), (60)]{zhan2}), and thus are different in general since the distances used are different. Since the differential equation of the adjoint vector $\psi$ depends on these multipliers, the existence of the adjoint vector in the main result \cite[Theorem 3.1]{zhan2} cannot be established.

A second serious mistake is in the use of the directional convexity assumption (see \cite[Equations (35), (36), (43)]{zhan2}). The first equality in (35) can obviously fail: the term $u^\varepsilon_{\omega,\lambda} (\tau)$ is a convex combination of $u^\varepsilon (\tau)$ and  $u_{\mu(\tau)}^\varepsilon$ since $\V(\tau)$ is assumed to be convex, but the parameter of this convex combination is not necessarily equal to $\lambda$ as claimed by the authors (unless $f$ is affine in $u$ and $f^0$ is convex in $u$, but this restrictive assumption is not made).
The nasty consequence of this error is that, in (43), the limit as $\lambda$ tends to $0$ is not valid.

A third mistake is in \cite[(57), (60), (62)]{zhan2}, when the authors claim that the rest of the proof can be led for dense points similarly as for right-scattered points. They pass to the limit in (60) as $\varepsilon$ tends to $0$ and get that $V^\varepsilon (b)$ tends to $V(b)$, where $V^\varepsilon$ is defined by (57) and $V$ is defined similarly. However, this does not hold true. Indeed, even though $d^*_\DD (u^\varepsilon,u^*)$ (Ekeland's distance) tends to $0$, there is no guarantee that $u^\varepsilon (\tau)$ tends to $u^*(\tau)$.

The above mistakes are major and cannot be corrected even through a major revision of the overall proof, due to evident obstructions. There are many other minor ones along the paper (which can be corrected, although some of them require a substantial work), such as:
the $\DD$-measurability of the map $\V$ is not proved;
in (45) the authors should consider subsequences and not a global limit;
in (55), any arbitrary $\rho>0$ cannot be considered to deal with the $\DD$-Lebesgue point $\tau$, but only with $\tau-\rho \in \T$ (recall that the equality \eqref{eq9999} of our paper is valid only if $s+\beta \in \T$, and that, as already mentioned, on a general time scale Lebesgue points must be handled with some special care).

In view of these numerous issues, it cannot be considered that the PMP has been proved in \cite{zhan2}. The aim of the present article (whose work was initiated far before we discovered the publication \cite{zhan2}) is to fill a gap in the literature and to derive a general strong version of the PMP on time scales.
Finally, it can be noted that the authors of \cite{zhan2} make restrictive assumptions: their set $\Omega$ is convex and is compact at scattered points, their dynamics are globally Lipschitzian and directionally convex, and they consider optimal control problems with fixed final time and fixed initial and final points. In the present article we go far beyond these unnecessary and not natural requirements, as already explained throughout.
\end{remark}

\subsection{Needle-like variations of admissible controls}\label{section3}
Let $b \in \T \backslash \{a\}$.
Following the definition of an admissible control (see Definition \ref{defadm} in Section \ref{section2}), we denote by $\UU$ the set of all $(u,q_a)\in\U\times\R^n$ such that $u$ is an admissible control on $[a,b]_\T$ associated with the initial condition $q_a$. It is endowed with the distance
\begin{equation}\label{def_dUU}
d_{\UU} ((u,q_a), (u',q'_a)) = \Vert u - u' \Vert_{\L^1_\T ([a,b[_\T,\R^{m})} + \Vert q_a - q'_a \Vert_{\R^n}.
\end{equation}

Throughout the section, we consider $(u,q_a) \in \UU$ with $u \in \L^\infty_\T ([a,b[_\T;\Omega)$ and the corresponding solution $q(\cdot,u,q_a)$ of \eqref{DD-CS} with $q(a)=q_a$. 
This section \ref{section3} is devoted to define appropriate variations of $(u,q_a)$, instrumental in order to prove the PMP.
We present some preliminary topological results in Section \ref{section30}. Then we define needle-like variations of $u$ in Sections \ref{section32} and \ref{section33}, respectively at a right-scattered point and at a right-dense point and derive some useful properties. Finally in Section \ref{section34} we make some variations of the initial condition $q_a$.

\subsubsection{Preliminaries}\label{section30}
In the first lemma below, we prove that $\UU$ is open. Actually we prove a stronger result, by showing that $\UU$ contains a neighborhood of any of its point in $\L^1$ topology, which will be useful in order to define needle-like variations.

\begin{lemma}\label{prop30-1}
Let $R > \Vert u \Vert_{\L^\infty_\T([a,b[_\T,\R^m)}$. 
There exist $\nu_R > 0$ and $\eta_R > 0$ such that the set
\begin{equation*}
\begin{split}
\E(u,q_a,R) = \{ (u',q'_a) \in \U\times\R^n   \ \vert \  & \Vert u'-u \Vert_{\L^1_\T([a,b[_\T,\R^m)} \leq \nu_R, \\ & \Vert u' \Vert_{\L^\infty_\T([a,b[_\T,\R^m)} \leq R, \; \Vert q'_a - q_a \Vert_{\R^n} \leq \eta_R \}
\end{split}
\end{equation*}
is contained in $\UU$.
\end{lemma}

Before proving this lemma, let us recall a time scale version of Gronwall's Lemma (see \cite[Chapter 6.1]{bohn}). 
The generalized exponential function is defined by
$e_L (t,c) = \exp ( \int_{[c,t[_\T} \xi_{\mu (\tau)} (L) \; \DD \tau )$,
for every $L \geq 0$, every $c \in \T$ and every $t \in [c,+\infty[_\T$,
where
$\xi_{\mu (\tau)} (L) = \log ( 1+ L \mu (\tau)) / \mu (\tau)$ whenever $\mu (\tau) > 0$, and $\xi_{\mu (\tau)} (L) = L$ whenever $\mu (\tau) =0$
(see \cite[Chapter 2.2]{bohn}). Note that, for every $L \geq 0$ and every $c \in \T$, the function $e_L (\cdot,c)$ (resp. $e_L (c,\cdot)$) is positive and increasing on $[c,+\infty[_\T$ (resp. positive and decreasing on $[a,c]_\T$), and moreover there holds $e_L (t_2,t_1) e_L (t_1,c) = e_L (t_2,c)$, for every $L \geq 0$ and all $(c, t_1, t_2)\in \T^3$ such that $c \leq t_1 \leq t_2$.

\begin{lemma}[\cite{bohn}]\label{lemgronwall}
Let $(c, d)\in\T^2$ such that $c < d$, let $L_1$ and $L_2$ be nonnegative real numbers, and let $q \in \CC ([c,d]_\T, \R )$ satisfying
$0 \leq q(t) \leq L_1 + L_2  \int_{[c,t[_\T} q(\tau) \,\DD \tau$,
for every $t \in [c,d]_\T$.
Then $0 \leq q(t) \leq L_1 e_{L_2} (t,c)$, for every $t \in [c,d]_\T$.
\end{lemma}

\begin{proof}[Proof of Lemma \ref{prop30-1}]
Let $R > \Vert u \Vert_{\L^\infty_\T([a,b[_\T,\R^m)}$. By continuity of $q(\cdot,u,q_a)$ on $[a,b]_\T$, the set
$$
K = \{ (x,v,t) \in \R^{n} \times \B_{\R^m}(0,R) \times [a,b]_\T\ \vert \ \Vert x-q(t,u,q_a) \Vert_{\R^n} \leq 1 \}
$$
is a compact subset of $\R^n \times \R^m \times \T$. Therefore $ \Vert \partial f / \partial q \Vert $ and $ \Vert \partial f / \partial u \Vert $ are bounded by some $L \geq 0$ on $K$ and moreover $L$ is chosen such that
\begin{equation}\label{eqlipK}
\Vert f(x_1,v_1,t) - f(x_2,v_2,t) \Vert_{\R^n}  \leq L ( \Vert x_1 -x_2 \Vert_{\R^n}  + \Vert v_1 -v_2 \Vert_{\R^m}  ),
\end{equation}
for all $(x_1,v_1,t)$ and $(x_2,v_2,t)$ in $K$. 
Let $\nu_R >0$ and $0 < \eta_R <1$ such that $ (\eta_R + \nu_R L) e_L (b,a) < 1$. Note that $K$, $L$, $\nu_R$ and $\eta_R$ depend on $ (u,q_a,R)$.

Let $(u',q'_a) \in \E(u,q_a,R)$. We denote by $I'_\T$ the interval of definition of $q(\cdot,u',q'_a)$ satisfying $a \in I'_\T$ and $I'_\T \backslash \{a \} \neq \emptyset$. It suffices to prove that $b \in I'_\T$. 
By contradiction, assume that the set $A = \{ t \in I'_\T \cap [a,b]_\T\ \vert \ \Vert q(t,u',q'_a) - q(t,u,q_a) \Vert_{\R^n}  > 1 \}$ is not empty and set $t_1 = \inf A$. Since $\T$ is closed, $t_1 \in I'_\T \cap [a,b]_\T$ and $[a,t_1]_\T \subset I'_\T \cap [a,b]_\T$. If $t_1$ is a minimum then $ \Vert q(t_1,u',q'_a) - q(t_1,u,q_a) \Vert_{\R^n} > 1$. If $t_1$ is not a minimum then $t_1 \in \SS$ and by continuity we have $ \Vert q(t_1,u',q'_a) - q(t_1,u,q_a) \Vert_{\R^n}  \geq 1$. Moreover there holds $t_1 > a$ since $\Vert q(a,u',q'_a) - q(a,u,q_a) \Vert_{\R^n}  = \Vert q'_a - q_a \Vert_{\R^n}  \leq \eta_R < 1$. Hence $\Vert q(\tau,u',q'_a) - q(\tau,u,q_a) \Vert_{\R^n}  \leq 1$ for every $\tau \in [a,t_1[_\T$. Therefore $(q(\tau,u',q'_a),u'(\tau),\tau)$ and $(q(\tau,u,q_a),u(\tau),\tau)$ are elements of $K$ for $\DD$-a.e.\ $\tau \in [a,t_1[_\T$.
Since there holds
$$
q(t,u',q'_a) - q(t,u,q_a) = q'_a - q_a  + \int_{[a,t[_\T} ( f(q(\tau,u',q'_a),u'(\tau),\tau) - f(q(\tau,u,q_a),u(\tau),\tau) ) \, \DD \tau ,
$$
for every $t \in I'_\T \cap [a,b]_\T$, it follows from \eqref{eqlipK} and from Lemma \ref{lemgronwall} that, for every $t \in [a,t_1]_\T$,
\begin{equation*}
\begin{split}
\Vert q(t,u',q'_a) - q(t,u,q_a) \Vert_{\R^n}  & \leq 
\Vert q'_a - q_a \Vert_{\R^n}  + L  \int_{[a,t[_\T} \Vert  u'(\tau) - u(\tau) \Vert_{\R^m}  \, \DD \tau \\
& \qquad + L  \int_{[a,t[_\T} \Vert   q(\tau,u',q'_a) - q(\tau,u,q_a) \Vert_{\R^n}  \, \DD \tau  \\
& \leq  (\Vert q'_a - q_a \Vert_{\R^n}  + L \Vert u'- u \Vert_{\L^1_\T([a,b[_\T,\R^m)}) e_L (b,a) \\
& \leq  (\eta_R + \nu_R L) e_L (b,a)  <  1.
\end{split}
\end{equation*}
This raises a contradiction at $t=t_1$. Therefore $A$ is empty and thus $q(\cdot,u',q'_a)$ is bounded on $I'_\T \cap [a,b]_\T$. It follows from \cite[Theorem 2]{bour10} that $b \in I'_\T$, that is, $(u',q'_a) \in \UU$.
\end{proof}

\begin{remark}\label{rmK}
Let $(u',q'_a) \in \E(u,q_a,R)$. With the notations of the above proof, since $I'_\T \cap [a,b]_\T = [a,b]_\T$ and $A$ is empty, we infer that $\Vert q(t,u',q'_a) - q(t,u,q_a) \Vert \leq 1$, for every $t \in [a,b]_\T$.
Therefore $(q(t,u',q'_a),u'(t),t) \in K$ for every $(u',q'_a) \in \E(u,q_a,R)$ and for $\DD$-a.e.\ $t \in [a,b[_\T$.
\end{remark}

\begin{lemma}\label{prop30-1-1}
With the notations of Lemma \ref{prop30-1}, the mapping
\begin{equation*}
\fonction{F_{(u,q_a,R)}}{(\E(u,q_a,R),d_{\UU})}{(\CC([a,b]_\T,\R^{n}),\Vert \cdot \Vert_\infty)}{(u',q'_a)}{q(\cdot,u',q'_a)}
\end{equation*}
 is Lipschitzian.
In particular, for every $(u',q'_a) \in \E(u,q_a,R)$, $q(\cdot,u',q'_a)$ converges uniformly to $q(\cdot,u,q_a)$ on $[a,b]_\T$ when $u'$ tends to $u$ in $\L^1_\T ([a,b[_\T,\R^m)$ and $q'_a $ tends to $q_a$ in $\R^n$.
\end{lemma}

\begin{proof}
Let $(u',q'_a)$ and $(u'',q''_a)$ be elements of $\E(u,q_a,R) \subset \UU$. 
It follows from Remark \ref{rmK} that $(q(\tau,u'',q''_a),u''(\tau),\tau)$ and $(q(\tau,u',q'_a),u'(\tau),\tau)$ are elements of $K$ for $\DD$-a.e.\ $\tau \in [a,b[_\T$. Following the same arguments as in the previous proof, it follows from \eqref{eqlipK} and from Lemma \ref{lemgronwall} that, for every $t \in [a,b]_\T$,
\begin{equation*}
\Vert q(t,u'',q''_a) - q(t,u',q'_a) \Vert_{\R^n}  \leq  (\Vert q''_a - q'_a \Vert_{\R^n}  + L \Vert u'' - u' \Vert_{\L^1_\T([a,b[_\T,\R^m)}) e_L (b,a).
\end{equation*}
The lemma follows.
\end{proof}

\subsubsection{Needle-like variation of $u$ at a right-scattered point}\label{section32}
Let $r \in [a,b[_\T \cap \RR$ and let $y \in \D (u(r))$. We define the needle-like variation $\Pi = (r,y)$ of $u$ at the right-scattered point $r$ by
\begin{equation*}
u_\Pi (t,\alpha) = \left\{ \begin{array}{lcl}
u(r) + \alpha (y-u(r)) & \textrm{if} & t=r, \\
u(t) & \textrm{if} & t \neq r .
\end{array} \right.
\end{equation*}
for every $\alpha \in \D(u(r),y)$.
It follows from Section~\ref{sec_topoprelim} that $u_\Pi (\cdot,\alpha) \in \L^\infty_{\T} ([a,b[_\T ; \Omega)$. 

\begin{lemma}\label{lem32-1}
There exists $\alpha_0 > 0$ such that $(u_\Pi (\cdot,\alpha),q_a) \in \UU$, for every $\alpha \in \D(u(r),y) \cap [0,\alpha_0]$.
\end{lemma}

\begin{proof}
Let $R = \max ( \Vert u \Vert_{\L^\infty_\T ([a,b[_\T,\R^m)},\Vert u(r) \Vert_{\R^m} + \Vert y \Vert_{\R^m} ) + 1 > \Vert u \Vert_{\L^\infty_\T ([a,b[_\T,\R^m)}$. We use the notations $K$, $L$, $\nu_R$ and $\eta_R$, associated with $(u,q_a,R)$, defined in Lemma \ref{prop30-1} and in its proof.

One has $\Vert u_\Pi (\cdot,\alpha) \Vert_{\L^\infty_\T ([a,b[_\T,\R^m)} \leq R$ for every $\alpha \in \D (u(r),y)$, and
$$
\Vert u_\Pi (\cdot,\alpha) - u \Vert_{\L^1_\T ([a,b[_\T,\R^m)} = \mu (r) \Vert u_\Pi (r,\alpha) - u(r) \Vert_{\R^m}  = \alpha \mu (r) \Vert y - u(r) \Vert_{\R^m} .
$$
Hence, there exists $\alpha_0 > 0$ such that $\Vert u_\Pi (\cdot,\alpha) - u \Vert_{\L^1_\T ([a,b[_\T,\R^m)} \leq \nu_R$ for every $\alpha \in \D (u(r),y) \cap [0,\alpha_0]$, and hence $(u_\Pi (\cdot,\alpha),q_a) \in \E(u,q_a,R)$. The claim follows then from Lemma~\ref{prop30-1}.
\end{proof}

\begin{lemma}\label{lem32-1-1}
The mapping
\begin{equation*}
\fonction{F_{(u,q_a,\Pi)}}{(\D(u(r),y) \cap [0,\alpha_0],\vert \cdot \vert)}{(\CC([a,b]_\T,\R^{n}),\Vert \cdot \Vert_\infty)}{\alpha}{q (\cdot,u_\Pi (\cdot,\alpha),q_a)}
\end{equation*}
is Lipschitzian.
In particular, for every $\alpha \in \D(u(r),y) \cap [0,\alpha_0]$, $q (\cdot,u_\Pi (\cdot,\alpha),q_a)$ converges uniformly to $q (\cdot,u,q_a)$ on $[a,b]_\T$ as $\alpha$ tends to $0$.
\end{lemma}

\begin{proof}
We use the notations of proof of Lemma~\ref{lem32-1}. It follows from Lemma~\ref{prop30-1-1} that there exists $C \geq 0$ (the Lipschitz constant of $F_{(u,q_a,R)}$) such that
\begin{equation*}
\begin{split}
\Vert q (\cdot,u_\Pi (\cdot,\alpha^2),q_a) - q (\cdot,u_\Pi (\cdot,\alpha^1),q_a) \Vert_\infty 
& \leq C d_{\UU} ( (u_\Pi (\cdot,\alpha^2),q_a), (u_\Pi (\cdot,\alpha^1),q_a) ) \\ 
& = C \vert \alpha^2 - \alpha^1 \vert \mu (r) \Vert y - u(r) \Vert_{\R^m} ,
\end{split}
\end{equation*}
for all $\alpha^1$ and $\alpha^2$ in $\D (u(r),y) \cap [0,\alpha_0]$.
The lemma follows.
\end{proof}

We define the so-called \textit{variation vector} $w_\Pi (\cdot,u,q_a)$ associated with the needle-like variation $\Pi = (r,y)$ as the unique solution on $[\sigma(r),b]_\T$ of the linear $\DD$-Cauchy problem
\begin{equation}\label{varvect_scattered}
w^\DD(t) = \frac{\partial f}{\partial q} (q (t,u(t),q_a),u(t),t) w(t) , \quad
w(\sigma(r)) = \mu (r)  \frac{\partial f}{\partial u} (q (r,u,q_a),u(r),r) (y -u(r)).
\end{equation}
The existence and uniqueness of $w_\Pi (\cdot,u,q_a)$ are ensured by \cite[Theorem 3]{bour10}.

\begin{proposition}\label{prop32-1}
The mapping
\begin{equation}
\fonction{F_{(u,q_a,\Pi)}}{(\D(u(r),y) \cap [0,\alpha_0],\vert \cdot \vert)}{(\CC ([\sigma(r),b]_\T,\R^{n}),\Vert \cdot \Vert_\infty)}{\alpha}{q (\cdot,u_\Pi (\cdot,\alpha),q_a)}
\end{equation}
is differentiable\footnote{Clearly this mapping can be extended to a neighborhood of $0$ and we speak of its differential at $0$ in this sense.} at $0$, and there holds $DF_{(u,q_a,\Pi)}(0) = w_\Pi (\cdot,u,q_a)$.
\end{proposition}

\begin{proof}
We use the notations of proof of Lemma~\ref{lem32-1}. Recall that $(q(t,u_\Pi(\cdot,\alpha),q_a),u_\Pi (t,\alpha),t) \in K$ for every $\alpha \in \D (u(r),y) \cap [0,\alpha_0] $ and for $\DD$-a.e.\ $t \in [a,b[_\T$, see Remark~\ref{rmK}. For every $\alpha \in \D (u(r),y) \cap ]0,\alpha_0]$ and every $t \in [\sigma (r),b]_\T$, we define
$$
\varepsilon_\Pi (t,\alpha) = \frac{q(t,u_\Pi(\cdot,\alpha),q_a) - q (t,u,q_a)}{\alpha }- w_{\Pi} (t,u,q_a).
$$
It suffices to prove that $\varepsilon_\Pi (\cdot,\alpha)$ converges uniformly to $0$ on $[\sigma (r),b]_\T$ as $\alpha$ tends to $0$.
For every $\alpha \in \D (u(r),y) \cap ]0,\alpha_0]$, the function $\varepsilon_\Pi (\cdot,\alpha)$ is absolutely continuous on $[\sigma(r),b]_\T$, and
$\varepsilon_\Pi (t,\alpha) = \varepsilon_\Pi (\sigma (r),\alpha) + \int_{[\sigma(r),t[_\T} \varepsilon^\DD_\Pi (\tau,\alpha) \, \DD \tau$
for every $t \in [\sigma (r),b]_\T$, where
\begin{equation*}
\varepsilon_\Pi^\DD (t,\alpha) = \frac{f(q(t,u_\Pi(\cdot,\alpha),q_a),u(t),t)-f(q (t,u,q_a),u(t),t)}{\alpha} - \frac{\partial f}{\partial q} (q(t,u,q_a),u(t),t) w_{\Pi} (t,u,q_a) ,
\end{equation*}
for $\DD$-a.e.\ $t \in [\sigma(r),b[_\T$. 
It follows from the Mean Value Theorem applied for $\DD$-a.e.\ $t \in [\sigma(r),b[_\T$ to the function defined by 
$\varphi_t(\theta) = f( (1-\theta) q(t,u,q_a) + \theta q(t,u_\Pi(\cdot,\alpha),q_a) , u(t), t)$ for every $\theta\in[0,1]$, 
that there exists $\theta_\Pi (t,\alpha)\in\R^n$, belonging to the segment of extremities $q(t,u,q_a)$ and $q(t,u_\Pi(\cdot,\alpha),q_a)$, such that
\begin{equation*}\begin{split}
\varepsilon_\Pi^\DD (t,\alpha) = & \frac{\partial f}{\partial q} (\theta_\Pi (t,\alpha),u(t),t) \varepsilon_\Pi (t,\alpha) \\ & + \left(\frac{\partial f}{\partial q} (\theta_\Pi (t,\alpha),u(t),t) - \frac{\partial f}{\partial q} (q(t,u,q_a),u(t),t) \right) w_{\Pi} (t,u,q_a) .
\end{split}\end{equation*}
Since $(\theta_\Pi (t,\alpha),u(t),t) \in K$ for $\DD$-a.e.\ $t \in [\sigma(r),b[_\T$, it follows that
$\Vert \varepsilon_\Pi^\DD (t,\alpha) \Vert \leq \chi_\Pi (t,\alpha) +  L \Vert \varepsilon_\Pi (t,\alpha) \Vert$, where
$
\chi_\Pi (t,\alpha) = \big\Vert \big(\frac{\partial f}{\partial q} (\theta_\Pi (t,\alpha),u(t),t) - \frac{\partial f}{\partial q} (q(t,u,q_a),u(t),t) \big)  w_{\Pi} (t,u,q_a) \big\Vert.
$
Therefore, one has
\begin{equation*}
\Vert \varepsilon_\Pi (t,\alpha) \Vert_{\R^n}  \leq \Vert \varepsilon_\Pi (\sigma(r),\alpha) \Vert_{\R^n}  +  \int_{[\sigma(r),b[_\T} \chi_\Pi (\tau,\alpha) \,\DD \tau  +L \int_{[\sigma(r),t[_\T} \Vert \varepsilon_\Pi (\tau,\alpha) \Vert_{\R^n}  \, \DD \tau,
\end{equation*}
for every $t \in [\sigma(r),b]_\T$. It follows from Lemma \ref{lemgronwall} that
$\Vert \varepsilon_\Pi (t,\alpha) \Vert_{\R^n}  \leq \Upsilon_\Pi (\alpha) e_L (b,\sigma(r))$, 
for every $t \in [\sigma(r),b]_\T$, where
$\Upsilon_\Pi (\alpha) = \Vert \varepsilon_\Pi (\sigma(r),\alpha) \Vert_{\R^n}  +  \int_{[\sigma(r),b[_\T} \chi_\Pi (\tau,\alpha) \, \DD \tau$.

To conclude, it remains to prove that $\Upsilon_\Pi (\alpha)$ converges to $0$ as $\alpha$ tends to $0$. 
First, since $\theta_\Pi (\cdot,\alpha)$ converges uniformly to $q(\cdot,u,q_a)$ on $[\sigma(r),b]_\T$ as $\alpha$ tends to $0$, and since $\partial f / \partial q$ is uniformly continuous on $K$, we infer that $\int_{[\sigma(r),b[_\T} \chi_\Pi (\tau,\alpha) \, \DD \tau$ converges to $0$ as $\alpha$ tends to $0$.  
Second, it is easy to see that $\Vert \varepsilon_\Pi (\sigma(r),\alpha) \Vert_{\R^n} $ converges to $0$ as $\alpha$ tends to $0$. The conclusion follows.
\end{proof}

\begin{lemma}\label{lem32-2}
Let $R > \Vert u \Vert_{\L^\infty_\T ([a,b[_\T,\R^m)}$ and let $(u_k,q_{a,k})_{k \in \N}$ be a sequence of elements of $\E(u,q_a,R)$. If $u_k$ converges to $u$ $\DD$-a.e.\ on $[a,b[_\T$ and $q_{a,k}$ converges to $q_a$ in $\R^n$ as $k$ tends to $+\infty$, then
$w_{\Pi}(\cdot,u_k,q_{a,k})$ converges uniformly to $w_\Pi(\cdot,u,q_a)$ on $[\sigma(r),b]_\T$ as $k$ tends to $+\infty$.
\end{lemma}

\begin{proof}
We use the notations $K$, $L$, $\nu_R$ and $\eta_R$, associated with $(u,q_a,R)$, defined in Lemma \ref{prop30-1} and in its proof.

Consider the absolutely continuous function defined by $\Phi_k(t) = w_{\Pi}(t,u_k,q_{a,k}) - w_{\Pi}(t,u,q_a)$ for every $k \in \N$ and every $t\in[\sigma(r),b]_\T$. Let us prove that $\Phi_k$ converges uniformly to $0$ on $[\sigma(r),b]_\T$ as $k$ tends to $+\infty$. One has
\begin{equation*}\begin{split}
\Phi_k (t) = & \Phi_k (\sigma(r) ) + \int_{[\sigma(r),t[_\T} \frac{\partial f}{\partial q} (q(\tau,u_k,q_{a,k}),u_k (\tau),\tau) \Phi_k (\tau) \, \DD \tau \\ 
& + \int_{[\sigma(r),t[_\T} \left(\frac{ \partial f}{\partial q} (q(\tau,u_k,q_{a,k}),u_k (\tau),\tau) - \frac{\partial f}{\partial q} (q(\tau,u,q_{a}),u (\tau),\tau) \right) w_{\Pi} (\tau,u,q_a) \, \DD \tau ,
\end{split}\end{equation*}
for every $t \in [\sigma(r),b]_\T$ and every $k \in \N$. Since $(u_k,q_{a,k}) \in \E(u,q_a,R)$ for every $k \in \N$, it follows from Remark~\ref{rmK} that $(q(t,u_k,q_{a,k}),u_k(t),t) \in K$ and $(q(t,u,q_a),u(t),t) \in K$ for $\DD$-a.e.\ $t \in [a,b[_\T$. Hence it follows from Lemma \ref{lemgronwall} that
$$\Vert \Phi_k (t) \Vert_{\R^n}   \leq  (\Vert \Phi_k (\sigma(r)) \Vert_{\R^n}  + \vartheta_k) e_L (b,\sigma(r)),$$
for every $t \in [\sigma(r),b]_\T$, where
$$\vartheta_k = \int_{[\sigma(r),b[_\T}  \left\Vert \frac{\partial f}{\partial q} (q(\tau,u_k,q_{a,k}),u_k (\tau),\tau) - \frac{\partial f}{\partial q} (q(\tau,u,q_a),u (\tau),\tau) \right\Vert_{\R^{n,n}} \Vert w_{\Pi} (\tau,u,q_a) \Vert_{\R^n}  \; \DD \tau .$$
Since $\mu_\DD (\{ r\} ) = \mu (r) > 0$, $u_k(r)$ converges to $u(r)$ as $k$ tends to $+\infty$. Moreover, $(u_k,q_{a,k})$ converges to $(u,q_a)$ in $(\E(u,q_a,R),d_{\UU})$ and, from Lemma \ref{prop30-1-1}, $q(\cdot,u_k,q_{a,k})$ converges uniformly to $q(\cdot,u,q_a)$ on $[a,b]_\T$ as $k$ tends to $+\infty$. We infer that $\Phi_k (\sigma (r) )$ converges to $0$ as $k$ tends to $+\infty$, and from the Lebesgue dominated convergence theorem we conclude that $\vartheta_k$ converges to $0$ as $k$ tends to $+\infty$. The lemma follows.
\end{proof}

\begin{remark}
It is interesting to note that, since $u_k (r)$ converges to $u(r)$ as $k$ tends to $+\infty$, if we assume that $y \in \DS(u(r))$, then $y \in \D(u_k(r))$ for $k$ sufficiently large.
\end{remark}

\subsubsection{Needle-like variation of $u$ at a right-dense point}\label{section33}
The definition of a needle-like variation at a Lebesgue right-dense point is very similar to the classical continuous-time case.
Let $s \in \LL_{[a,b[_\T} (f(q(\cdot,u,q_a),u(\cdot),\cdot)) \cap \SS$ and $z \in \Omega$. We define the needle-like variation $\amalg = (s,z)$ of $u$ at $s$ by
\begin{equation*}
u_\amalg (t,\beta) = \left\{ \begin{array}{lcl}
z & \textrm{if} & t \in [s,s+\beta[_\T , \\
u(t) & \textrm{if} & t \notin [s,s+\beta[_\T.
\end{array} \right.
\end{equation*}
for every $\beta \in \V^b_s$ (here, we use the notations introduced in Section \ref{section1}).
Note that $u_\amalg (\cdot,\beta) \in \L^\infty_{\T} ([a,b[_\T;\Omega)$. 

\begin{lemma}\label{lem33-1}
There exists $\beta_0 > 0$ such that $(u_\amalg (\cdot,\beta),q_a) \in \UU$ for every $\beta \in \V^b_s \cap [0,\beta_0]$.
\end{lemma}

\begin{proof}
Let $R = \max ( \Vert u \Vert_{\L^\infty_\T ([a,b[_\T,\R^m)},\Vert z \Vert_{\R^m}) + 1 > \Vert u \Vert_{\L^\infty_\T ([a,b[_\T,\R^m)}$. We use the notations $K$, $L$, $\nu_R$ and $\eta_R$, associated with $(u,q_a,R)$, defined in Lemma \ref{prop30-1} and in its proof.

For every $\beta \in \V^b_s$ one has $\Vert u_\amalg (\cdot,\beta) \Vert_{\L^\infty_\T ([a,b[_\T,\R^m)} \leq R$ and
$$
\Vert u_\amalg (\cdot,\beta) - u \Vert_{\L^1_\T ([a,b[_\T,\R^m)} = \int_{[s,s+\beta[_\T} \Vert z - u(\tau) \Vert_{\R^m}  \, \DD \tau \leq 2 R \beta.
$$
Hence, there exists $\beta_0 > 0$ such that for every $\beta \in \V^b_s \cap [0,\beta_0]$, $\Vert u_\amalg (\cdot,\beta) - u \Vert_{\L^1_\T ([a,b[_\T,\R^m)} \leq \nu_R$ and thus $(u_\amalg (\cdot,\beta),q_a) \in \E(u,q_a,R)$. The conclusion then follows from Lemma \ref{prop30-1}.
\end{proof}

\begin{lemma}\label{lem33-1-1}
The mapping
\begin{equation*}
\fonction{F_{(u,q_a,\amalg)}}{(\V^b_s \cap [0,\beta_0],\vert \cdot \vert)}{(\CC([a,b]_\T,\R^{n}),\Vert \cdot \Vert_\infty)}{\beta}{q (\cdot,u_\amalg (\cdot,\beta),q_a)}
\end{equation*}
is Lipschitzian.
In particular, for every $\beta \in \V^b_s \cap \B(0,\beta_0)$, $q (\cdot,u_\amalg (\cdot,\beta),q_a)$ converges uniformly to $q(\cdot,u,q_a)$ on $[a,b]_\T$ as $\beta$ tends to $0$.
\end{lemma}

\begin{proof}
We use the notations of proof of Lemma~\ref{lem33-1}. From Lemma \ref{prop30-1-1}, there exists $C \geq 0$ (Lipschitz constant of $F_{(u,q_a,R)}$) such that 
\begin{equation*}\begin{split}
\Vert q (\cdot,u_\amalg (\cdot,\beta^2),q_a) - q (\cdot,u_\amalg (\cdot,\beta^1),q_a) \Vert_\infty 
& \leq C d_{\UU} ((u_\amalg (\cdot,\beta^2),q_a),(u_\amalg (\cdot,\beta^1),q_a)) \\
& \leq 2 C R \vert \beta^2-\beta^1 \vert,
\end{split}\end{equation*}
for all $\beta^1$ and $\beta^2$ in $\V^b_s \cap [0,\beta_0]$.
The lemma follows.
\end{proof}

According to \cite[Theorem 3]{bour10}, we define the \textit{variation vector} $w_{\amalg}(\cdot,u,q_a)$ associated with the needle-like variation $\amalg = (s,z)$ as the unique solution on $[s,b]_\T$ of the linear $\DD$-Cauchy problem
\begin{equation}\label{varvect_dense}
w^\DD(t)  = \frac{\partial f}{\partial q} (q(t,u,q_a),u(t),t)w(t) , \quad w(s) = f(q(s,u,q_a),z,s) - f(q(s,u,q_a),u(s),s).
\end{equation}

\begin{proposition}\label{prop33-1}
For every $\delta \in \V^b_s \backslash \{ 0 \}$, the mapping
\begin{equation}
\fonction{F^\delta_{(u,q_a,\amalg)}}{(\V^b_s \cap [0,\beta_0],\vert \cdot \vert)}{(\CC ([s+\delta,b]_\T,\R^{n}),\Vert \cdot \Vert_\infty)}{\beta}{q (\cdot,u_\amalg (\cdot,\beta),q_a)}
\end{equation}
is differentiable at $0$, and one has $ DF^\delta_{(u,q_a,\amalg)}(0) = w_\amalg (\cdot,u,q_a)$.
\end{proposition}

\begin{proof}
We use the notations of proof of Lemma~\ref{lem33-1}. Recall that $(q (t,u_\amalg (\cdot,\beta),q_a),u_\amalg (t,\beta),t)$ and $(q (t,u_\amalg (\cdot,\beta),q_a),z,t)$ belong to $K$ for every $\beta \in \V^b_s \cap [0,\beta_0]$ and for $\DD$-a.e.\ $t \in [a,b[_\T$, see Remark~\ref{rmK}. For every $\beta \in \V^b_s \cap ]0,\beta_0]$ and every $t \in [s+\beta,b]_\T$, we define
$$ \varepsilon_\amalg (t,\beta) = \frac{q (t,u_\amalg (\cdot,\beta),q_a) - q(t,u,q_a)}{ \beta } - w_{\amalg} (t,u,q_a). $$
It suffices to prove that $\varepsilon_\amalg (\cdot,\beta)$ converges uniformly to $0$ on $[s+\beta,b]_\T$ as $\beta$ tends to $0$ (note that, for every $\delta \in \V^b_s \backslash \{ 0 \}$, it suffices to consider $\beta \leq \delta$).
For every $\beta \in \V^b_s \cap ]0,\beta_0]$, the function $\varepsilon_\amalg (\cdot,\beta)$ is absolutely continuous on $[s+\beta,b]_\T$ and
$\varepsilon_\amalg (t,\beta) = \varepsilon_\amalg (s+\beta,\beta) +  \int_{[s+\beta,t[_\T} \varepsilon^\DD_\amalg (\tau,\beta) \, \DD \tau$,
for every $t \in [s+\beta,b]_\T$, 
where
$$
\varepsilon_\amalg^\DD (t,\beta) = \frac{f(q (t,u_\amalg (\cdot,\beta),q_a),u(t),t)-f(q (t,u,q_a),u(t),t)}{\beta}  - \frac{\partial f}{\partial q} (q(t,u,q_a),u(t),t)w_{\amalg} (t,u,q_a) .
$$
for $\DD$-a.e.\ $t \in [s+\beta,b[_\T$.
As in the proof of Proposition \ref{prop32-1}, it follows from the Mean Value Theorem that, for $\DD$-a.e.\ $t \in [s+\beta,b[_\T$, there exists $\theta_\amalg (t,\beta)\in\R^n$, belonging to the segment of extremities $q(t,u,q_a)$ and $q (t,u_\amalg (\cdot,\beta),q_a)$, such that
\begin{equation*}\begin{split}
\varepsilon_\amalg^\DD (t,\beta) = & \frac{\partial f}{\partial q} (\theta_\amalg (t,\beta),u(t),t)\varepsilon_\amalg (t,\beta) \\ & + \left(\frac{\partial f}{\partial q} (\theta_\amalg (t,\beta),u(t),t) - \frac{\partial f}{\partial q} (q(t,u,q_a),u(t),t) \right) w_{\amalg} (t,u,q_a).
\end{split}\end{equation*}
Since $(\theta_\amalg (t,\beta),u(t),t) \in K$ for $\DD$-a.e.\ $t \in [s+\beta,b[_\T$, it follows that
$\Vert \varepsilon_\amalg^\DD (t,\beta) \Vert \leq \chi_\amalg (t,\beta) +  L \Vert \varepsilon_\amalg (t,\beta) \Vert$,
where
$
\chi_\amalg (t,\beta) = \big\Vert \big(\frac{\partial f}{\partial q} (\theta_\amalg (t,\beta),u(t),t) - \frac{\partial f}{\partial q} (q(t,u,q_a),u(t),t) \big)w_{\amalg} (t,u,q_a) \big\Vert.
$
Therefore, one has
\begin{equation*}
\Vert \varepsilon_\amalg (t,\beta) \Vert_{\R^n}  \leq \Vert \varepsilon_\amalg (s+\beta,\beta) \Vert_{\R^n}  +  \int_{[s+\beta,b[_\T} \chi_\amalg (\tau,\beta) \; \DD \tau  +L  \int_{[s+\beta,t[_\T} \Vert \varepsilon_\amalg (\tau,\beta) \Vert_{\R^n}  \, \DD \tau ,
\end{equation*}
for every $t \in [s+\beta,b]_\T$, and it follows from Lemma \ref{lemgronwall} that
$\Vert \varepsilon_\amalg (t,\beta) \Vert \leq \Upsilon_\amalg (\beta) e_L (b,s)$,
for every $t \in [s+\beta,b]_\T$, where
$ \Upsilon_\amalg (\beta) = \Vert \varepsilon_\amalg (s+\beta,\beta) \Vert_{\R^n}  +  \int_{[s+\beta,b[_\T} \chi_\amalg (\tau,\beta) \, \DD \tau$.

To conclude, it remains to prove that $\Upsilon_\amalg (\beta)$ converges to $0$ as $\beta$ tends to $0$. First, since $\theta_\amalg (\cdot,\beta)$ converges uniformly to $q(\cdot,u,q_a)$ on $[s+\beta,b]_\T$ as $\beta$ tends to $0$ and since $\partial f / \partial q$ is uniformly continuous on $K$, we infer that $\int_{[s+\beta,b[_\T} \chi_\amalg (\tau,\beta) \, \DD \tau$ converges to $0$ as $\beta$ tends to $0$.
Second, let us prove that $\Vert \varepsilon_\amalg (s+\beta,\beta) \Vert_{\R^n} $ converges to $0$ as $\beta$ tends to $0$. By continuity, $w_\amalg (s+\beta,u,q_a)$ converges to $w_\amalg (s,u,q_a)$ as $\beta$ to $0$. Moreover, since $q (\cdot,u_\amalg (\cdot,\beta),q_a)$ converges uniformly to $q(\cdot,u,q_a)$ on $[a,b]_\T$ as $\beta$ tends to $0$ and since $f$ is uniformly continuous on $K$, it follows that $f(q (\cdot,u_\amalg (\cdot,\beta),q_a),z,t)$ converges uniformly to $f(q(\cdot,u,q_a),z,t)$ on $[a,b]_\T$ as $\beta$ tends to $0$. Therefore, it suffices to note that
$$\frac{1}{ \beta } \int_{[s,s+\beta[_\T} f(q (\tau,u,q_a),z,\tau) - f(q(\tau,u,q_a),u(\tau),\tau) \, \DD \tau$$
converges to $w_\amalg (s,u,q_a) = f(q(s,u,q_a),z,s) - f(q(s,u,q_a),u(s),s)$ as $\beta$ tends to $0$ since $s$ is a $\DD$-Lebesgue point of $f(q(\cdot,u,q_a),z,t)$ and of $f(q(\cdot,u,q_a),u,t)$. 
Then $\Vert \varepsilon_\amalg (s+\beta,\beta) \Vert$ converges to $0$ as $\beta$ tends to $0$, and hence $\Upsilon_\amalg (\beta)$ converges to $0$ as well.
\end{proof}

\begin{lemma}\label{lem33-2}
Let $R > \Vert u \Vert_{\L^\infty_\T ([a,b[_\T,\R^m)}$ and let $(u_k,q_{a,k})_{k \in \N}$ be a sequence of elements of $\E(u,q_a,R)$. If $u_k$ converges to $u$ $\DD$-a.e.\ on $[a,b[_\T$ and $q_{a,k}$ converges to $q_a$ as $k$ tends to $+\infty$, then $w_{\amalg}(\cdot,u_k,q_{a,k})$ converges uniformly to $w_\amalg(\cdot,u,q_a)$ on $[s,b]_\T$ as $k$ tends to $+\infty$.
\end{lemma}

\begin{proof}
The proof is similar to the one of Lemma \ref{lem33-1-1}, replacing $\sigma(r)$ with $s$.
\end{proof}

\subsubsection{Variation of the initial condition $q_a$}\label{section34}
Let $q'_a \in \R^n$. 

\begin{lemma}\label{lem34-1}
There exists $\gamma_0>0$ such that $(u,q_a+\gamma q'_a) \in \UU$ for every $\gamma \in [0,\gamma_0]$.
\end{lemma}

\begin{proof}
Let $R = \Vert u \Vert_{\L^\infty_\T ([a,b[_\T,\R^m)} + 1 > \Vert u \Vert_{\L^\infty_\T ([a,b[_\T,\R^m)}$. We use the notations $K$, $L$, $\nu_R$ and $\eta_R$, associated with $(u,q_a,R)$, defined in Lemma \ref{prop30-1} and in its proof.

There exists $\gamma_0>0$ such that $\Vert q_a+\gamma q'_a - q_a \Vert_{\R^n}  = \gamma \Vert q'_a \Vert_{\R^n}  \leq \eta_R$ for every $\gamma \in [0,\gamma_0]$, and hence $(u,q_a+\gamma q'_a) \in \E(u,q_a,R)$. Then the claim follows from Lemma \ref{prop30-1}.
\end{proof}

\begin{lemma}\label{lem34-1-1}
The mapping
\begin{equation*}
\fonction{F_{(u,q_a,q'_a)}}{([0,\gamma_0],\vert \cdot \vert)}{(\CC([a,b]_\T,\R^{n}),\Vert \cdot \Vert_\infty)}{\gamma}{q (\cdot,u,q_a+\gamma q'_a)}
\end{equation*}
is Lipschitzian.
In particular, for every $\gamma \in [0,\gamma_0]$, $q (\cdot,u,q_a+\gamma q'_a)$ converges uniformly to $q (\cdot,u,q_a)$ on $[a,b]_\T$ as $\gamma$ tends to $0$.
\end{lemma}

\begin{proof}
We use the notations of proof of Lemma~\ref{lem34-1}. From Lemma \ref{prop30-1}, there exists $C \geq 0$ (Lipschitz constant of $F_{(u,q_a,R)}$) such that 
\begin{equation*}\begin{split}
\Vert q (\cdot,u,q_a + \gamma^2 q'_a ) - q (\cdot,u,q_a + \gamma^1 q'_a ) \Vert_\infty 
& \leq  C d_{\UU}((u,q_a+\gamma^2 q'_a),(u,q_a+\gamma^1 q'_a)) \\
& = C \vert \gamma^2 - \gamma^1 \vert \Vert q'_a \Vert_{\R^n} .
\end{split}\end{equation*}
for all $\gamma^1$ and $\gamma^2$ in $[0,\gamma_0]$.
\end{proof}

According to \cite[Theorem 3]{bour10}, we define the \textit{variation vector} $w_{q'_a}(\cdot,u,q_a)$ associated with the perturbation $q_a'$ as the unique solution on $[a,b]_\T$ of the linear $\DD$-Cauchy problem
\begin{equation}\label{varvect_pointinit}
w^\DD(t) = \frac{\partial f}{\partial q} (q (t,u,q_a),u(t),t)w(t), \quad w(a) = q'_a.
\end{equation}

\begin{proposition}\label{prop34-1}
The mapping
\begin{equation}
\fonction{F_{(u,q_a,q'_a)}}{([0,\gamma_0],\vert \cdot \vert)}{(\CC ([a,b]_\T,\R^{n}),\Vert \cdot \Vert_\infty)}{\gamma}{q (\cdot,u,q_a+\gamma q'_a )}
\end{equation}
is differentiable at $0$, and one has $DF_{(u,q_a,q'_a)}(0) = w_{q'_a} (\cdot,u,q_a)$.
\end{proposition}

\begin{proof}
We use the notations of proof of Lemma~\ref{lem33-1}. Note that, from Remark~\ref{rmK}, $(q (t,u,q_a+\gamma q'_a),u(t),t) \in K$ for every $\gamma \in [0,\gamma_0]$ and for $\DD$-a.e.\ $t \in [a,b[_\T$.
For every $\gamma \in ]0,\gamma_0]$ and every $t \in [a,b]_\T$, we define
\begin{equation*}
\varepsilon_{q'_a} (t,\gamma) = \frac{q (t,u,q_a+\gamma q'_a) - q(t,u,q_a)}{ \gamma } - w_{q'_a} (t,u,q_a).
\end{equation*}
It suffices to prove that $\varepsilon_{q'_a} (\cdot,\gamma)$ converges uniformly to $0$ on $[a,b]_\T$ as $\gamma$ tends to $0$.
For every $\gamma \in ]0,\gamma_0]$, the function $\varepsilon_{q'_a} (\cdot,\gamma)$ is absolutely continuous on $[a,b]_\T$, and
$\varepsilon_{q'_a} (t,\gamma) = \varepsilon_{q'_a} (a,\gamma) +  \int_{[a,t[_\T} \varepsilon^\DD_{q'_a} (\tau,\gamma) \, \DD \tau$, for every $t \in [a,b]_\T$, where
\begin{equation*}
\varepsilon_{q'_a}^\DD (t,\gamma) = \frac{f(q (t,u,q_a+\gamma q'_a),u(t),t)-f(q (t,u,q_a),u(t),t)}{\gamma} - \frac{\partial f}{\partial q} (q(t,u,q_a),u(t),t)w_{q'_a} (t,u,q_a) ,
\end{equation*}
for $\DD$-a.e.\ $t \in [a,b[_\T$.
As in the proof of Proposition \ref{prop32-1}, it follows from the Mean Value Theorem that, for $\DD$-a.e.\ $t \in [a,b[_\T$, there exists $\theta_{q'_a} (t,\gamma) \in\R^n$, belonging to the segment of extremities $q(t,u,q_a)$ and $q (t,u,q_a+\gamma q'_a)$, such that
\begin{equation*}\begin{split}
\varepsilon_{q'_a}^\DD (t,\gamma) =&  \frac{\partial f}{\partial q} (\theta_{q'_a} (t,\gamma),u(t),t)\varepsilon_{q'_a} (t,\gamma)  \\ & + \left(\frac{\partial f}{\partial q} (\theta_{q'_a} (t,\gamma),u(t),t) - \frac{\partial f}{\partial q} (q(t,u,q_a),u(t),t) \right) w_{q'_a} (t,u,q_a).
\end{split}\end{equation*}
Since $(\theta_{q'_a} (t,\gamma),u(t),t) \in K$ for $\DD$-a.e.\ $t \in [a,b[_\T$, it follows that
\begin{equation*}
\Vert \varepsilon_{q'_a}^\DD (t,\gamma) \Vert_{\R^n}  \leq \chi_{q'_a} (t,\gamma) +  L \Vert \varepsilon_{q'_a} (t,\gamma) \Vert_{\R^n} ,
\end{equation*}
where
$
\chi_{q'_a} (t,\gamma) = \big\Vert \big(\frac{\partial f}{\partial q} (\theta_{q'_a} (t,\gamma),u(t),t) - \frac{\partial f}{\partial q} (q(t,u,q_a),u(t),t) \big) \times w_{q'_a} (t,u,q_a) \big\Vert_{\R^n} .
$
Hence
\begin{equation*}
\Vert \varepsilon_{q'_a} (t,\gamma) \Vert_{\R^n}  \leq \Vert \varepsilon_{q'_a} (a,\gamma) \Vert_{\R^n}  +  \int_{[a,b[_\T} \chi_{q'_a} (\tau,\gamma) \, \DD \tau  +L \int_{[a,t[_\T} \Vert \varepsilon_{q'_a} (\tau,\gamma) \Vert_{\R^n}  \, \DD \tau,
\end{equation*}
for every $t \in [a,b]_\T$, and it follows from Lemma \ref{lemgronwall} that
$\Vert \varepsilon_{q'_a} (t,\gamma) \Vert \leq \Upsilon_{q'_a} (\gamma) e_L (b,a)$, 
for every $t \in [a,b]_\T$, where
$
\Upsilon_{q'_a} (\gamma) = \Vert \varepsilon_{q'_a} (a,\gamma) \Vert_{\R^n}  +  \int_{[a,b[_\T} \chi_{q'_a} (\tau,\gamma) \, \DD \tau.
$

To conclude, it remains to prove that $\Upsilon_{q'_a} (\gamma)$ converges to $0$ as $\gamma$ tends to $0$.
First, since $\theta_{q'_a} (\cdot,\gamma)$ converges uniformly to $q(\cdot,u,q_a)$ on $[a,b]_\T$ as $\gamma$ tends to $0$ and since $\partial f / \partial q$ is uniformly continuous on $K$, we infer that $\int_{[a,b[_\T} \chi_{q'_a} (\tau,\gamma) \; \DD \tau$ tends to $0$ when $\gamma \to 0$. Second, it is easy to see that $\varepsilon_{q'_a} (a,\gamma) = 0$ for every $\gamma \in ]0,\gamma_0]$. The conclusion follows.
\end{proof}

\begin{lemma}\label{lem34-2}
Let $R > \Vert u \Vert_{\L^\infty_\T([a,b[_\T,\R^m)}$ and let $(u_k,q_{a,k})_{k \in \N}$ be a sequence of elements of $\E(u,q_a,R)$. If $u_k$ converges to $u$ $\DD$-a.e.\ on $[a,b[_\T$ and $q_{a,k}$ converges to $q_a$ in $\R^n$ as $k$ tends to $+\infty$, then $w_{q'_a}(\cdot,u_k,q_{a,k})$ converges uniformly to $w_{q'_a}(\cdot,u,q_a)$ on $[a,b]_\T$ as $k$ tends to $+\infty$.
\end{lemma}

\begin{proof}
The proof is similar to the one of Lemma \ref{lem33-1-1}, replacing $\sigma(r)$ with $a$.
\end{proof}

\subsection{Proof of the PMP}\label{section4}
Throughout this section we consider $\OCP$ with a fixed final time $b \in \T \backslash \{ a\}$. 
We proceed as is very usual (see e.g.\ \cite{lee,pont}) by considering the \textit{augmented control system} in $\R^{n+1}$
\begin{equation} \label{DD-CS_augm}
\bar q^\DD(t) = \bar f (\bar q(t),u(t),t),
\end{equation}
with $\bar q=(q,q^0)^{\mathrm{T}}\in \R^n\times\R$, the augmented state, and $\bar f:\R^{n+1}\times\R^m\times\T\rightarrow\R^{n+1}$, the augmented dynamics, defined by 
$\bar f(\bar q,u,t)=(f(q,u,t),f^0(q,u,t))^{\mathrm{T}}$. The additional coordinate $q^0$ stands for the cost, and we will always impose as an initial condition $q^0(a)=0$, so that $q^0(b) = C(b,u) = \int_{[a,b[_\T} f^0 (q (\tau), u(\tau), \tau ) \, \DD \tau$.
The function $\bar g : \R^{n+1} \times \R^{n+1}\rightarrow \R^j$ is defined by $\bar g(\bar q_1,\bar q_2) = g(q_1,q_2)$, where $\bar q_i=(q_i,q^0_i)$ for $i=1,2$.
Note that $\bar f$ does not depend on $q^0$ and that $\bar{g}$ does not depend on $q_1^0$ nor on $q_2^0$. Note as well that the Hamiltonian of $\OCP$ is written as $H(q,u,p,p^0,t)=\langle\bar{p},\bar{f}(\bar{q},u,t)\rangle_{\R^{n+1}}$.

With these notations, $\OCP$ consists of determining a trajectory $\bar q^*(\cdot)=(q^*(\cdot),q^{0*}(\cdot))$ defined on $[a,b]_\T$, solution of \eqref{DD-CS_augm} and associated with a control $u^*\in \L^\infty_{\T}([a,b[_\T;\Omega)$, minimizing 
$q^0(b)$
over all possible trajectories $\bar q(\cdot)=(q(\cdot),{q^0}(\cdot))$ defined on $[a,b]_\T$, solutions of \eqref{DD-CS_augm} and associated with an admissible control $u\in \L^\infty_{\T}([a,b[_\T;\Omega)$ and satisfying $\bar g(\bar q(a),\bar q(b))\in \S$.

In what follows, let $\bar q^*(\cdot)$ be such an optimal trajectory. Set $q_a^*=q^*(a)$. We are going to apply first Ekeland's Variational Principle to a well chosen functional in an appropriate complete metric space, and then, using needle-like variations as defined previously (applied to the augmented system, that is, with the dynamics $\bar f$), we are going to derive some inequalities, finally resulting into the desired statement of the PMP.

\subsubsection{Application of Ekeland's Variational Principle}\label{section41}
For the completeness, we recall Ekeland's Variational Principle.
\begin{theorem}[\cite{ekel}]\label{lemekeland}
Let $(\E,d_\E)$ be a complete metric space and $J:\E\rightarrow\R \cup \{ +\infty \} $ be a lower semi-continuous function which is bounded below. Let $\varepsilon > 0$ and $u^* \in \E$ such that $J(u^*) \leq \inf_{u \in \E} J(u) + \varepsilon$.
Then there exists $u_\varepsilon \in \E$ such that $d_\E (u_\varepsilon,u^*) \leq \sqrt{\varepsilon}$ and
$J(u_\varepsilon) \leq J(u)+\sqrt{\varepsilon} d_\E (u,u_\varepsilon)$ for every $u \in \E$.
\end{theorem}

Recall from Lemma \ref{prop30-1} that, for $R > \Vert u^* \Vert_{\L^\infty_\T([a,b[_\T,\R^m)}$, the set $\E(u^*,\bar{q}^*_a,R)$ defined in this lemma is contained in $\UU$. To take into account the set $\Omega$ of constraints on the controls, we define
$$
\E_\Omega^R = \{ (u,\bar{q}_a) \in \U\times\R^{n+1} \ \vert \  \bar{q}_a=(q_a,0), \ (u,q_a)\in \E(u^*,\bar{q}^*_a,R),\ u\in \L^\infty_\T ([a,b[_\T;\Omega) \}.
$$
Using the fact that $\Omega$ is closed\footnote{Note that the assumption $\Omega$ closed is used (only) here in a crucial way. In the proof of the classical continuous-time PMP this assumption is not required because the Ekeland distance which is then used is defined by $\rho(u,v)=\mu_L(\{t\in[a,b]\ \vert\ u(t)\neq v(t)\}$), and obviously the set of measurable functions $u:[a,b]\rightarrow\Omega$ endowed with this distance is complete, under the sole assumption that $\Omega$ is measurable. In the discrete-time setting and a fortiori in the general time scale setting, this distance cannot be used any more. Here we use the distance $d_{\UU}$ defined by \eqref{def_dUU} but then to ensure completeness it is required to assume that $\Omega$ is closed.\label{footnote_Omegaclosed}}, it clearly follows from the (partial) converse of Lebesgue's Dominated Convergence Theorem that $(\E_\Omega^R,d_{\UU})$ is a complete metric space. 

Before applying Ekeland's Variational Principle in this space, let us introduce several notations and recall basic facts in order to handle the convex set $\S$.
We denote by $d_\S$ the distance function to $\S$ defined by $d_\S (x) = \inf_{x' \in \S} \Vert x-x' \Vert_{\R^j}$, for every $x\in\R^j$.
Recall that, for every $x \in \R^j$, there exists a unique element $\P_\S (x)\in\S$ (projection of $x$ onto $\S$) such that $d_\S (x) = \Vert x - \P_\S(x) \Vert_{\R^j}$. It is characterized by the property
$\langle x-\P_\S (x) , x'-\P_\S (x) \rangle_{\R^j} \leq 0$ for every $x'\in \S$.
Moreover, the projection mapping $\P_\S$ is $1$-Lipschitz continuous. Furthermore, there holds $x-\P_\S (x) \in \O_\S (\P_\S (x))$ for every $x \in \R^j$ (where $\O_\S (x)$ is defined by \eqref{defortho}). We recall the following obvious lemmas.

\begin{lemma}\label{lemconvex}
Let $(x_k)_{k \in \N}$ be a sequence of points of $\R^j$ and $(\zeta_k)_{k \in \N}$ be a sequence of nonnegative real numbers such that $x_k \to x \in \S$ and $\zeta_k (x_k - \P_\S (x_k) ) \to x' \in \R^j$ as $k \to +\infty$.
Then $x' \in \O_\S (x)$.
\end{lemma}


\begin{lemma}
The function $x \mapsto d_\S^2 (x)^2$ is differentiable on $\R^j$, and $\mathrm{d} d_\S^2 (x) \cdot x'= 2\langle x-\P_\S (x), x' \rangle_{\R^j}$.
\end{lemma}

We are now in a position to apply Ekeland's Variational Principle.
For every $\varepsilon > 0$ such that $\sqrt{\varepsilon} < \min ( \nu_R, \eta_R )$, we consider the   functional $J^R_\varepsilon:(\E^R_\Omega,d_{\UU})\rightarrow \R^+$ defined by
$$
J^R_\varepsilon(u,\bar{q}_a) = \left( \max(q^0(b,u,\bar{q}_a) - q^{0*}(b) + \varepsilon,0)^2 + d^2_{\S} (\bar{g}(\bar{q}_a,\bar{q}(b,u,\bar{q}_a))) \right)^{1/2}.
$$
Since $F_{(u^*,\bar{q}^*_a,R)}$ (by Lemma~\ref{prop30-1-1}), $\bar{g}$ and $d_{\S}$ are continuous, it follows that $J^R_\varepsilon$ is continuous on $(\E^R_\Omega,d_{\UU})$. Moreover, one has $J^R_\varepsilon (u^*,\bar{q}^*_a) = \varepsilon$ and $J^R_\varepsilon (u,\bar{q}_a) > 0$ for every $(u,\bar{q}_a) \in \E^R_\Omega$.
It follows from Ekeland's Variational Principle that, for every $\varepsilon > 0$ such that $\sqrt{\varepsilon} < \min ( \nu_R, \eta_R )$, there exists $(u^R_\varepsilon,\bar{q}^R_{a,\varepsilon}) \in \E^R_\Omega$ such that $d_{\UU} ( (u^R_\varepsilon,\bar{q}^R_{a,\varepsilon}),(u^*,\bar{q}^*_a) ) \leq \sqrt{\varepsilon}$ and
\begin{equation}\label{eqconsequenceekeland}
-\sqrt{\varepsilon} d_{\UU} ((u,\bar{q}_a) , (u^R_\varepsilon,\bar{q}^R_{a,\varepsilon})) \leq  J^R_\varepsilon (u,\bar{q}_a) - J^R_\varepsilon (u^R_\varepsilon,\bar{q}^R_{a,\varepsilon}),
\end{equation}
for every $(u,\bar{q}_a) \in \E^R_\Omega$.
In particular, $u^R_\varepsilon$ converges to $u^*$ in $\L^1_\T([a,b[_\T,\R^m)$ and $\bar{q}^R_{a,\varepsilon}$ converges to $\bar{q}^*_a$ as $\varepsilon$ tends to $0$. Besides, setting
\begin{equation}\label{defp0R}
\psi^{0R}_\varepsilon = \frac{-1}{J^R_\varepsilon (u^R_\varepsilon,\bar{q}^R_{a,\varepsilon})} \max(q^0 (b,u^R_\varepsilon,\bar{q}^R_{a,\varepsilon}) - q^{0*}(b) + \varepsilon,0) \leq 0
\end{equation}
and
\begin{equation}\label{defpsiepsR}
\psi^R_\varepsilon = \frac{-1}{J^R_\varepsilon (u^R_\varepsilon,\bar{q}^R_{a,\varepsilon})} \left( \bar{g} (\bar{q}^R_{a,\varepsilon},\bar{q}(b,u^R_\varepsilon,\bar{q}^R_{a,\varepsilon})) - \P_{\S} (\bar{g} (\bar{q}^R_{a,\varepsilon},\bar{q}(b,u^R_\varepsilon,\bar{q}^R_{a,\varepsilon}))) \right) \in \R^{j} ,
\end{equation}
note that $\vert \psi^{0R}_\varepsilon \vert^2 + \Vert \psi^R_\varepsilon \Vert_{\R^j}^2 = 1 $ and $-\psi^R_\varepsilon \in \O_\S (\P_{\S} (\bar{g} (\bar{q}^R_{a,\varepsilon},\bar{q}(b,u^R_\varepsilon,\bar{q}^R_{a,\varepsilon}))))$. 

Using a compactness argument, the continuity of $F_{(u^*,\bar{q}^*_a,R)}$ and the $\CC^1$ regularity of $\bar g$, and the (partial) converse of the Dominated Convergence Theorem, we infer that there exists a sequence $(\varepsilon_k)_{k \in \N}$ of positive real numbers converging to $0$ such that $u^R_{\varepsilon_k}$ converges to $u^*$ $\DD$-a.e.\ on $[a,b[_\T$,
$\bar{q}^R_{a,\varepsilon_k}$ converges to $\bar{q}^*_a$,
$\bar{g}(\bar{q}^R_{a,\varepsilon_k},\bar{q}(b,u^R_{\varepsilon_k},\bar{q}^R_{a,\varepsilon_k}))$ converges to $\bar{g}(\bar{q}^*_a,\bar{q}^*(b)) \in \S$,
$\mathrm{d}\bar{g}(\bar{q}^R_{a,\varepsilon_k},\bar{q}(b,u^R_{\varepsilon_k},\bar{q}^R_{a,\varepsilon_k}))$ converges to $\mathrm{d}\bar{g}(\bar{q}^*_a,\bar{q}^*(b))$,
$\psi^{0R}_{\varepsilon_k}$ converges to some $\psi^{0R} \leq 0$,
and $\psi^R_{\varepsilon_k}$ converges to some $\psi^R \in \R^j$ as $k$ tends to $+\infty$,
with $\vert \psi^{0R} \vert^2 + \Vert \psi^R \Vert_{\R^j}^2 = 1 $ and $-\psi^R \in \O_{\S} (\bar{g} (\bar{q}^*_a,\bar{q}^*(b)))$ (see Lemma~\ref{lemconvex}).

\medskip

In the three next lemmas, we use the inequality \eqref{eqconsequenceekeland} respectively with needle-like variations of $u^R_{\varepsilon_k}$ at right-scattered points and then at right-dense points, and variations of $\bar{q}^R_{a,\varepsilon_k}$, and infer some crucial inequalities by taking the limit in $k$.
Note that these variations were defined in Section \ref{section3} for any dynamics $f$, and that we apply them here to the augmented system \eqref{DD-CS_augm}, associated with the augmented dynamics $\bar f$.

\begin{lemma}\label{lemfondamscattered}
For every $r \in [a,b[_\T \cap \RR$ and every $y \in \DS(u^*(r))$, considering the needle-like variation $\Pi = (r,y)$ at the right-scattered point $r$ as defined in Section \ref{section32}, there holds
\begin{equation}\label{inegfondamscattered}
\psi^{0R}  w^{0}_\Pi (b,u^*,\bar{q}^*_a) + \Big\langle \Big( \frac{\partial \bar{g}}{\partial q_2} (\bar{q}^*_a,\bar{q}^*(b)) \Big)^{\!\mathrm{T}}  \psi^R , w_\Pi (b,u^*,\bar{q}^*_{a}) \Big\rangle_{\R^n} \leq 0 ,
\end{equation}
where the variation vector $\bar w_\Pi=(w_\Pi,w^0_\Pi)$ is defined by \eqref{varvect_scattered} (replacing $f$ with $\bar f$).
\end{lemma}

\begin{proof}
Since $u^R_{\varepsilon_k}$ converges to $u^*$ $\DD$-a.e. on $[a,b[_\T$, it follows that $u^R_{\varepsilon_k} (r)$ converges to $u^*(r)$ as $k$ tends to $+\infty$. Hence $y \in \D (u^R_{\varepsilon_k}(r))$ and $\Vert u^R_{\varepsilon_k} (r) \Vert_{\R^m} < R$ for $k$ sufficiently large. Fixing such a large integer $k$, we recall that $u^R_{\varepsilon_k,\Pi} (\cdot,\alpha) \in \L^\infty_\T ([a,b[_\T;\Omega)$ for every $\alpha \in \D (u^R_{\varepsilon_k}(r),y)$, and
\begin{equation*}\begin{split}
\Vert u^R_{\varepsilon_k,\Pi} (\cdot,\alpha) \Vert_{\L^\infty_\T ([a,b[_\T,\R^m)}
& \leq \max (\Vert u^R_{\varepsilon_k} \Vert_{\L^\infty_\T ([a,b[_\T,\R^m)}, \Vert u^R_{\varepsilon_k,\Pi} (r,\alpha) \Vert_{\R^m} ) \\ 
& \leq \max (R, \Vert u^R_{\varepsilon_k} (r) \Vert_{\R^m} + \alpha \Vert y - u^R_{\varepsilon_k} (r) \Vert_{\R^m}  ) ,
\end{split}\end{equation*}
and
\begin{equation*}\begin{split}
\Vert u^R_{\varepsilon_k,\Pi} (\cdot,\alpha) - u^* \Vert_{\L^1_\T ([a,b[_\T,\R^m)} 
& \leq \Vert u^R_{\varepsilon_k,\Pi} (\cdot,\alpha) - u^R_{\varepsilon_k} \Vert_{\L^1_\T ([a,b[_\T,\R^m)} + \Vert u^R_{\varepsilon_k}-u^* \Vert_{\L^1_\T ([a,b[_\T,\R^m)} \\ 
& \leq \alpha \mu (r) \Vert y-u^R_{\varepsilon_k} (r) \Vert_{\R^m} + \sqrt{\varepsilon_k}.
\end{split}\end{equation*}
Therefore $(u^R_{\varepsilon_k,\Pi} (\cdot,\alpha),\bar{q}^R_{a,\varepsilon_k}) \in \E^R_\Omega$ for every $\alpha \in \D (u^R_{\varepsilon_k}(r),y)$ sufficiently small.
It then follows from \eqref{eqconsequenceekeland} that
\begin{equation*}
-\sqrt{\varepsilon_k} \Vert u^R_{\varepsilon_k,\Pi} (\cdot,\alpha) - u^R_{\varepsilon_k} \Vert_{\L^1_\T ([a,b[_\T,\R^m)} \leq  J^R_{k} (u^R_{\varepsilon_k,\Pi} (\cdot,\alpha),\bar{q}^R_{a,\varepsilon_k}) - J^R_{k} (u^R_{\varepsilon_k},\bar{q}^R_{a,\varepsilon_k}),
\end{equation*}
and thus
\begin{equation*}
-\sqrt{\varepsilon_k} \mu (r) \Vert y-u^R_{\varepsilon_k} (r) \Vert_{\R^m} 
\leq \frac{J^R_{k} (u^R_{\varepsilon_k,\Pi} (\cdot,\alpha),\bar{q}^R_{a,\varepsilon_k})^2 - J^R_{k} (u^R_{\varepsilon_k},\bar{q}^R_{a,\varepsilon_k})^2}{\alpha ( J^R_{k} (u^R_{\varepsilon_k,\Pi} (\cdot,\alpha),\bar{q}^R_{a,\varepsilon_k}) + J^R_{k} (u^R_{\varepsilon_k},\bar{q}^R_{a,\varepsilon_k}) ) } .
\end{equation*}
Using 
Proposition~\ref{prop32-1}, since $\bar{g}$ does not depend on $q_2^0$, we infer that
\begin{equation*}\begin{split}
& \lim_{\alpha \to 0} \frac{J^R_{k} (u^R_{\varepsilon_k,\Pi} (\cdot,\alpha),\bar{q}^R_{a,\varepsilon_k})^2 - J^R_{k} (u^R_{\varepsilon_k},\bar{q}^R_{a,\varepsilon_k})^2}{\alpha} \\
&= 2 \max(q^0 (b,u^R_{\varepsilon_k},\bar{q}^R_{a,\varepsilon_k}) - {q^0}^*(b) + \varepsilon_k,0)  w^{0}_\Pi (b,u^R_{\varepsilon_k},\bar{q}^R_{a,\varepsilon_k}) \\ 
& \quad
+ 2 \Big\langle \bar{g}(\bar{q}^R_{a,\varepsilon_k},\bar{q}(b,u^R_{\varepsilon_k},\bar{q}^R_{a,\varepsilon_k})) - \P_{\S} (\bar{g}(\bar{q}^R_{a,\varepsilon_k},\bar{q}(b,u^R_{\varepsilon_k},\bar{q}^R_{a,\varepsilon_k}))), \\
& \qquad\qquad\qquad\qquad\qquad      \frac{\partial \bar{g}}{\partial q_2} (\bar{q}^R_{a,\varepsilon_k},\bar{q}(b,u^R_{\varepsilon_k},\bar{q}^R_{a,\varepsilon_k}))w_\Pi (b,u^R_{\varepsilon_k},\bar{q}^R_{a,\varepsilon_k}) \Big\rangle_{\R^j}.
\end{split}\end{equation*}
Since $J^R_{k} (u^R_{\varepsilon_k,\Pi} (\cdot,\alpha),\bar{q}^R_{a,\varepsilon_k})$ converges to $J^R_{k} (u^R_{\varepsilon_k},\bar{q}^R_{a,\varepsilon_k})$ as $\alpha$ tends to $0$, using \eqref{defp0R} and \eqref{defpsiepsR} it follows that
\begin{multline*}
-\sqrt{\varepsilon_k} \mu (r) \Vert y-u^R_{\varepsilon_k} (r) \Vert_{\R^m} \leq -\psi^{0R}_{\varepsilon_k}  w^{0}_\Pi (b,u^R_{\varepsilon_k},\bar{q}^R_{a,\varepsilon_k}) \\
- \Big\langle \Big( \frac{\partial \bar{g}}{\partial q_2} (\bar{q}^R_{a,\varepsilon_k},\bar{q}(b,u^R_{\varepsilon_k},\bar{q}^R_{a,\varepsilon_k})) \Big)^{\!\mathrm{T}} \psi^R_{\varepsilon_k} , w_\Pi (b,u^R_{\varepsilon_k},\bar{q}^R_{a,\varepsilon_k}) \Big\rangle_{\R^n}.
\end{multline*}
By letting $k$ tend to $+\infty$, and using Lemma~\ref{lem32-2}, the lemma follows.
\end{proof}

Denote by $\LL^R_{[a,b[_\T}$ the set of (Lebesgue) times $t \in [a,b[_\T$ such that $t \in \LL_{[a,b[_\T} (f(q(\cdot,u^*,q^*_a),u^*,t))$, such that $t\in \LL_{[a,b[_\T} (f(q(\cdot,u^R_{\varepsilon_k},q^R_{a,k}),u^R_{\varepsilon_k},t))$ for every $k\in\N$, and such that $u^R_{\varepsilon_k} (t)$ converges to $u^*(t)$ as $k$ tends to $+\infty$.
There holds $\mu_\DD (\LL^R_{[a,b[_\T}) = \mu_\DD ([a,b[_\T) = b-a$. 

\begin{lemma}\label{lemfondamdense}
For every $s \in \LL^R_{[a,b[_\T} \cap \SS$ and any $z \in \Omega \cap \B_{\R^m}(0,R)$, 
considering the needle-like variation $\amalg = (s,z)$ as defined in Section \ref{section33}, there holds
\begin{equation}\label{inegfondamdense}
\psi^{0R}  w^{0}_\amalg (b,u^*,\bar{q}^*_a) + \Big\langle \Big( \frac{\partial \bar{g}}{\partial q_2} (\bar{q}^*_a,\bar{q}^*(b)) \Big)^{\!\mathrm{T}}  \psi^R , w_\amalg (b,u^*,\bar{q}^*_a) \Big\rangle_{\R^n} \leq 0 ,
\end{equation}
where the variation vector $\bar w_\amalg=(w_\amalg,w^0_\amalg)$ is defined by \eqref{varvect_dense} (replacing $f$ with $\bar f$).
\end{lemma}

\begin{proof}
For every $k \in \N$ and any $\beta \in \V^b_s$, we recall that $u^R_{\varepsilon_k,\amalg} (\cdot,\beta) \in \L^\infty_\T ([a,b[_\T,\Omega) $ and
\begin{equation*}
\Vert u^R_{\varepsilon_k,\amalg} (\cdot,\beta) \Vert_{\L^\infty_\T ([a,b[_\T,\R^m)} \leq \max (\Vert u^R_{\varepsilon_k} \Vert_{\L^\infty_\T ([a,b[_\T,\R^m)} , \Vert z \Vert_{\R^m} ) \leq R  ,
\end{equation*}
and
\begin{equation*}\begin{split}
\Vert u^R_{\varepsilon_k,\amalg} (\cdot,\beta) - u^* \Vert_{\L^1_\T ([a,b[_\T,\R^m)} 
& \leq \Vert u^R_{\varepsilon_k,\amalg} (\cdot,\beta) - u^R_{\varepsilon_k} \Vert_{\L^1_\T ([a,b[_\T,\R^m)} + \Vert u^R_{\varepsilon_k}-u^* \Vert_{\L^1_\T ([a,b[_\T,\R^m)} \\ 
& \leq 2R \beta + \sqrt{\varepsilon_k}.
\end{split}\end{equation*}
Therefore $(u^R_{\varepsilon_k,\amalg} (\cdot,\beta),\bar{q}^R_{a,\varepsilon_k}) \in \E^R_\Omega$ for $\beta \in \V^b_s$ sufficiently small.
It then follows from \eqref{eqconsequenceekeland} that
\begin{equation*}
-\sqrt{\varepsilon_k} \Vert u^R_{\varepsilon_k,\amalg} (\cdot,\beta) - u^R_{\varepsilon_k} \Vert_{\L^1_\T ([a,b[_\T,\R^m)} \leq  J^R_{k} (u^R_{\varepsilon_k,\amalg} (\cdot,\beta),\bar{q}^R_{a,\varepsilon_k}) - J^R_{k} (u^R_{\varepsilon_k},\bar{q}^R_{a,\varepsilon_k}) ,
\end{equation*}
and thus
\begin{equation*}
-2 R \sqrt{\varepsilon_k} 
\leq \frac{J^R_{k} (u^R_{\varepsilon_k,\amalg} (\cdot,\beta),\bar{q}^R_{a,\varepsilon_k})^2 - J^R_{k} (u^R_{\varepsilon_k},\bar{q}^R_{a,\varepsilon_k})^2}{\beta ( J^R_{k} (u^R_{\varepsilon_k,\amalg} (\cdot,\beta),\bar{q}^R_{a,\varepsilon_k}) + J^R_{k} (u^R_{\varepsilon_k},\bar{q}^R_{a,\varepsilon_k}) ) } .
\end{equation*}
Using Proposition~\eqref{prop33-1}, since $\bar{g}$ does not depend on $q_2^0$, we infer that
\begin{equation*}\begin{split}
& \lim_{\beta \to 0} \frac{J^R_{k} (u^R_{\varepsilon_k,\amalg} (\cdot,\beta),\bar{q}^R_{a,\varepsilon_k})^2 - J^R_{k} (u^R_{\varepsilon_k},\bar{q}^R_{a,\varepsilon_k})^2}{\beta} \\ 
&= 2 \max(q^0 (b,u^R_{\varepsilon_k},\bar{q}^R_{a,\varepsilon_k}) - {q^0}^*(b) + \varepsilon_k,0) w^{0}_\amalg (b,u^R_{\varepsilon_k},\bar{q}^R_{a,\varepsilon_k}) \\ 
& \quad + 2 \Big\langle \bar{g}(\bar{q}^R_{a,\varepsilon_k},\bar{q}(b,u^R_{\varepsilon_k},\bar{q}^R_{a,\varepsilon_k})) - \P_{\S} (\bar{g}(\bar{q}^R_{a,\varepsilon_k},\bar{q}(b,u^R_{\varepsilon_k},\bar{q}^R_{a,\varepsilon_k}))), \\
& \qquad\qquad\qquad\qquad\qquad    \frac{\partial \bar{g}}{\partial q_2}(\bar{q}^R_{a,\varepsilon_k},\bar{q}(b,u^R_{\varepsilon_k},\bar{q}^R_{a,\varepsilon_k}))w_\amalg (b,u^R_{\varepsilon_k},\bar{q}^R_{a,\varepsilon_k}) \Big\rangle_{\R^j}.
\end{split}\end{equation*}
Since $J^R_{k} (u^R_{\varepsilon_k,\amalg} (\cdot,\beta),\bar{q}^R_{a,\varepsilon_k})$ converges to $J^R_{k} (u^R_{\varepsilon_k},\bar{q}^R_{a,\varepsilon_k})$ as $\alpha$ tends to $0$, using \eqref{defp0R} and \eqref{defpsiepsR} it follows that
\begin{equation*}
-2R \sqrt{\varepsilon_k} \leq -\psi^{0R}_{\varepsilon_k}  w^{0}_\amalg (b,u^R_{\varepsilon_k},\bar{q}^R_{a,\varepsilon_k}) 
- \Big\langle \Big( \frac{\partial \bar{g}}{\partial q_2} (\bar{q}^R_{a,\varepsilon_k},\bar{q}(b,u^R_{\varepsilon_k},\bar{q}^R_{a,\varepsilon_k})) \Big)^{\!\mathrm{T}}  \psi^R_{\varepsilon_k} , w_\amalg (b,u^R_{\varepsilon_k},\bar{q}^R_{a,\varepsilon_k}) \Big\rangle_{\R^n}.
\end{equation*}
By letting $k$ tend to $+\infty$, and using Lemma~\ref{lem33-2}, the lemma follows.
\end{proof}

\begin{lemma}\label{lemfondampoint}
For every $\bar{q}_a \in \R^n \times \{ 0 \}$, considering the variation of initial point as defined in Section \ref{section34}, there holds
\begin{equation}\label{inegfondampoint}
\psi^{0R}  w^0_{\bar{q}_a} (b,u^*,\bar{q}^*_a)+ \Big\langle \Big( \frac{\partial \bar{g}}{\partial q_2} (\bar{q}^*_a,\bar{q}^*(b)) \Big)^{\!\mathrm{T}}  \psi^R , w_{\bar{q}_a} (b,u^*,\bar{q}^*_a) \Big\rangle_{\R^n} 
\leq - \Big\langle \Big( \frac{\partial \bar{g}}{\partial q_1} (\bar{q}^*_a,\bar{q}^*(b)) \Big)^{\!\mathrm{T}}  \psi^R , q_a \Big\rangle_{\R^n}  ,
\end{equation}
where the variation vector $\bar w_{\bar{q}_a}=(w_{\bar{q}_a},w^0_{\bar{q}_a})$ is defined by \eqref{varvect_pointinit} (replacing $f$ with $\bar f$).
\end{lemma}

\begin{proof}
For every $k \in \N$ and every $\gamma\geq 0$, one has
$$
\Vert \bar{q}^R_{a,\varepsilon_k}+\gamma \bar{q}_a - \bar{q}^*_a \Vert_{\R^n} \leq  \gamma \Vert \bar{q}_a \Vert_{\R^n} + \Vert \bar{q}^R_{a,\varepsilon_k} - \bar{q}^*_a \Vert_{\R^n} \leq \gamma \Vert \bar{q}_a \Vert_{\R^n} + \sqrt{\varepsilon_k}.
$$
Therefore $(u^R_{\varepsilon_k},\bar{q}^R_{a,\varepsilon_k}+\gamma \bar{q}_a) \in \E^R_\Omega$ for $\gamma\geq 0$ sufficiently small.
It then follows from \eqref{eqconsequenceekeland} that
\begin{equation*}
-\sqrt{\varepsilon_k} \Vert \bar{q}^R_{a,\varepsilon_k}+\gamma \bar{q}_a - \bar{q}^R_{a,\varepsilon_k} \Vert_{\R^n} \leq  J^R_{k} (u^R_{\varepsilon_k},\bar{q}^R_{a,\varepsilon_k}+\gamma \bar{q}_a) - J^R_{k} (u^R_{\varepsilon_k},\bar{q}^R_{a,\varepsilon_k}),
\end{equation*}
and thus
\begin{equation*}
-\sqrt{\varepsilon_k} \Vert \bar{q}_a \Vert_{\R^n}  
\leq \frac{J^R_{k} (u^R_{\varepsilon_k},\bar{q}^R_{a,\varepsilon_k}+\gamma \bar{q}_a)^2 - J^R_{k} (u^R_{\varepsilon_k},\bar{q}^R_{a,\varepsilon_k})^2}{\gamma ( J^R_{k} (u^R_{\varepsilon_k},\bar{q}^R_{a,\varepsilon_k}+\gamma \bar{q}_a) + J^R_{k} (u^R_{\varepsilon_k},\bar{q}^R_{a,\varepsilon_k}) )  } .
\end{equation*}
Using Proposition~\eqref{prop34-1}, since $\bar{g}$ does not depend on $q_1^0$ and $q_2^0$, we infer that
\begin{equation*}\begin{split}
& \lim_{\gamma \to 0} \frac{J^R_{k} (u^R_{\varepsilon_k},\bar{q}^R_{a,\varepsilon_k}+\gamma \bar{q}_a)^2 - J^R_{k} (u^R_{\varepsilon_k},\bar{q}^R_{a,\varepsilon_k})^2}{\gamma} \\ 
& = 2 \max(q^0 (b,u^R_{\varepsilon_k},\bar{q}^R_{a,\varepsilon_k}) - {q^0}^*(b) + \varepsilon_k,0) w^{0}_{\bar{q}_a} (b,u^R_{\varepsilon_k},\bar{q}^R_{a,\varepsilon_k}) \\ 
& \quad + 2 \Big\langle \bar{g}(\bar{q}^R_{a,\varepsilon_k},\bar{q}(b,u^R_{\varepsilon_k},\bar{q}^R_{a,\varepsilon_k})) - \P_{\S} (\bar{g}(\bar{q}^R_{a,\varepsilon_k},\bar{q}(b,u^R_{\varepsilon_k},\bar{q}^R_{a,\varepsilon_k}))), \frac{\partial \bar{g}}{\partial q_1} (\bar{q}^R_{a,\varepsilon_k},\bar{q}(b,u^R_{\varepsilon_k},\bar{q}^R_{a,\varepsilon_k}))q_a \Big\rangle_{\R^j} \\ 
&  \quad + 2 \Big\langle \bar{g}(\bar{q}^R_{a,\varepsilon_k},\bar{q}(b,u^R_{\varepsilon_k},\bar{q}^R_{a,\varepsilon_k})) - \P_{\S} (\bar{g}(\bar{q}^R_{a,\varepsilon_k},\bar{q}(b,u^R_{\varepsilon_k},\bar{q}^R_{a,\varepsilon_k}))), \\
& \qquad\qquad\qquad\qquad\qquad   \frac{\partial \bar{g}}{\partial q_2} (\bar{q}^R_{a,\varepsilon_k},\bar{q}(b,u^R_{\varepsilon_k},\bar{q}^R_{a,\varepsilon_k}))w_{\bar{q}_a} (b,u^R_{\varepsilon_k},\bar{q}^R_{a,\varepsilon_k}) \Big\rangle_{\R^j}.
\end{split}\end{equation*}
Since $J^R_{k} (u^R_{\varepsilon_k},\bar{q}^R_{a,\varepsilon_k}+\gamma \bar{q}_a)$ converges to $J^R_{k} (u^R_{\varepsilon_k},\bar{q}^R_{a,\varepsilon_k})$ as $\gamma$ tends to $0$, using \eqref{defp0R} and \eqref{defpsiepsR} it follows that
\begin{multline*}
-\sqrt{\varepsilon_k} \Vert \bar{q}_a \Vert \leq -\psi^{0R}_{\varepsilon_k}  w^{0}_{\bar{q}_a} (b,u^R_{\varepsilon_k},\bar{q}^R_{a,\varepsilon_k}) - \Big\langle \Big( \frac{\partial \bar{g}}{\partial q_1} (\bar{q}^R_{a,\varepsilon_k},\bar{q}(b,u^R_{\varepsilon_k},\bar{q}^R_{a,\varepsilon_k})) \Big)^{\!\mathrm{T}}  \psi^R_{\varepsilon_k} , q_a \Big\rangle_{\R^n} \\ 
- \Big\langle \Big( \frac{\partial \bar{g}}{\partial q_2} (\bar{q}^R_{a,\varepsilon_k},\bar{q}(b,u^R_{\varepsilon_k},\bar{q}^R_{a,\varepsilon_k})) \Big)^{\!\mathrm{T}}  \psi^R_{\varepsilon_k} , w_{\bar{q}_a} (b,u^R_{\varepsilon_k},\bar{q}^R_{a,\varepsilon_k}) \Big\rangle_{\R^n}.
\end{multline*}
By letting $k$ tend to $+\infty$, and using Lemma~\ref{lem34-2}, the lemma follows.
\end{proof}

At this step, we have obtained in the three previous lemmas the three fundamental inequalities \eqref{inegfondamscattered}, \eqref{inegfondamdense} and \eqref{inegfondampoint}, valuable for any $R> \Vert u^* \Vert_{\L^\infty_\T ([a,b[_\T,\R^m)}$. Recall that $\vert \psi^{0R} \vert^2 + \Vert \psi^R \Vert_{\R^j}^2 = 1 $ and $-\psi^R \in \O_{\S} (\bar{g} (\bar{q}^*_a,\bar{q}^*(b)))$.
Then, considering a sequence of real numbers $R_\ell$ converging to $+\infty$ as $\ell$ tends to $+\infty$, we infer that there exist $\psi^0 \leq 0$ and $\psi \in \R^j$ such that $\psi^{0R_\ell}$ converges to $\psi^0$ and $\psi^{R_\ell}$ converges to $\psi$ as $\ell$ tends to $+\infty$, and moreover $\vert \psi^0 \vert^2 + \Vert \psi \Vert_{\R^j}^2 = 1 $ and $-\psi \in \O_\S (\bar{g} (\bar{q}^*_a,\bar{q}^*(b))) $ (since $\O_\S (\bar{g} (\bar{q}^*_a,\bar{q}^*(b)))$ is a closed subset of $\R^j$).

We set $\LL_{[a,b[_\T} = \bigcap_{\ell \in \N} \LL^{R_\ell}_{[a,b[_\T}$. Note that $\mu_\DD (\LL_{[a,b[_\T}) = \mu_\DD ([a,b[_\T) = b-a$. Taking the limit in $\ell$ in \eqref{inegfondamscattered}, \eqref{inegfondamdense} and \eqref{inegfondampoint}, we get the following lemma.

\begin{lemma}
We have the following variational inequalities.

For every $r \in [a,b[_\T \cap \RR$, and every $y \in \DS(u^*(r))$, there holds
\begin{equation}\label{varineqscattered}
\psi^0  w^{0}_\Pi (b,u^*,\bar{q}^*_a) + \Big\langle \Big( \frac{\partial \bar{g}}{\partial q_2}(\bar{q}^*_a,\bar{q}^*(b)) \Big)^{\!\mathrm{T}} \psi , w_\Pi (b,u^*,\bar{q}^*_a) \Big\rangle_{\R^n} \leq 0, 
\end{equation}
where the variation vector $\bar w_\Pi=(w_\Pi,w^0_\Pi)$ associated with the needle-like variation $\Pi = (r,y)$ is defined by \eqref{varvect_scattered} (replacing $f$ with $\bar f$);

For every $s \in \LL_{[a,b[_\T} \cap \SS$, and every $z \in \Omega$, there holds
\begin{equation}\label{varineqdense}
\psi^0  w^{0}_\amalg (b,u^*,\bar{q}^*_a) + \Big\langle \Big( \frac{\partial \bar{g}}{\partial q_2}(\bar{q}^*_a,\bar{q}^*(b)) \Big)^{\!\mathrm{T}}  \psi , w_\amalg (b,u^*,\bar{q}^*_a) \rangle_{\R^n} \leq 0,
\end{equation}
where the variation vector $\bar w_\amalg=(w_\amalg,w^0_\amalg)$ associated with the needle-like variation $\amalg = (s,z)$ is defined by \eqref{varvect_dense} (replacing $f$ with $\bar f$);

For every $\bar{q}_a \in \R^n \times \{ 0 \}$, there holds
\begin{equation}\label{varineqpointinit}
\psi^0  w^{0}_{\bar{q}_a} (b,u^*,\bar{q}^*_a)+ \Big\langle \Big( \frac{\partial \bar{g}}{\partial q_2}(\bar{q}^*_a,\bar{q}^*(b)) \Big)^{\!\mathrm{T}} \psi , w_{\bar{q}_a} (b,u^*,\bar{q}^*_a) \Big\rangle_{\R^n} \leq - \Big\langle \Big( \frac{\partial \bar{g}}{\partial q_1}(\bar{q}^*_a,\bar{q}^*(b)) \Big)^{\!\mathrm{T}}  \psi, q_a \Big\rangle_{\R^n},
\end{equation}
where the variation vector $\bar w_{\bar{q}_a}=(w_{\bar{q}_a},w^0_{\bar{q}_a})$ associated with the variation $\bar{q}_a$ of the initial point $q_a^*$ is defined by \eqref{varvect_pointinit} (replacing $f$ with $\bar f$).
\end{lemma}

This result concludes the application of Ekeland's Variational Principle.
The last step of the proof consists of deriving the PMP from these inequalities.

\subsubsection{End of the proof}\label{section46}
We define $\bar{p}(\cdot)=(p(\cdot),p^0(\cdot))$ as the unique solution on $[a,b]_\T$ of the backward shifted linear $\DD$-Cauchy problem
\begin{equation*}
\bar{p}^\DD(t) = - \Big( \frac{\partial \bar{f}}{\partial \bar{q}} (\bar{q}^*(t),u^*(t),t) \Big)^{\!\mathrm{T}} \bar{p}^\sigma(t), \quad
\bar{p}(b) = \Big( \Big( \frac{\partial \bar{g}}{\partial q_2} (\bar{q}^*_a,\bar{q}^*(b)) \Big)^{\!\mathrm{T}} \psi,\psi^0 \Big)^{\!\mathrm{T}}.
\end{equation*}
The existence and uniqueness of $\bar{p}(\cdot)$ are ensured by \cite[Theorem 6]{bour10}. Since $\bar f$ does not depend on $q^0$, it is clear that $p^0(\cdot)$ is constant, still denoted by $p^0$ (with $p^0=\psi^0$).

\paragraph{Right-scattered points.} Let $r \in [a,b[_\T \cap \RR$ and $y \in \DS(u^*(r))$. Since the function $t\mapsto \langle \bar{w}_\Pi(t,u^*,\bar{q}^*_a),\bar{p}(t) \rangle_{\R^{n+1}}$ is absolutely continuous, it holds $\langle \bar{p}(\cdot) , \bar{w}_\Pi(\cdot,u^*,\bar{q}^*_a) \rangle_{\R^{n+1}}^\DD = 0
$ $\DD$-almost everywhere on $[\sigma(r),b[_\T$ from the Leibniz formula \eqref{eqleibniz} and hence the function $\langle \bar{p}(\cdot) , \bar{w}_\Pi(\cdot,u^*,\bar{q}^*_a) \rangle_{\R^{n+1}}$ is constant on $[\sigma(r),b]_\T$. It thus follows from \eqref{varineqscattered} that
\begin{equation*}
\begin{split}
\langle \bar{p}(\sigma(r)) , \bar{w}_\Pi(\sigma(r),u^*,\bar{q}^*_a) \rangle_{\R^{n+1}} 
& = \langle \bar{p}(b) , \bar{w}_\Pi(b,u^*,\bar{q}^*_a) \rangle_{\R^{n+1}} \\ 
& = \psi^0  w^{0}_\Pi (b,u^*,\bar{q}^*_a) + \Big\langle \Big( \frac{\partial \bar{g}}{\partial q_2} (\bar{q}^*_a,\bar{q}^*(b)) \Big)^{\!\mathrm{T}}  \psi , w_\Pi (b,u^*,\bar{q}^*_a) \Big\rangle_{\R^n} \leq 0 ,
\end{split}
\end{equation*}
and since $\bar{w}_\Pi(\sigma(r),u^*,\bar{q}^*_a) = \mu (r) \frac{\partial \bar{f}}{\partial u} (\bar{q}^*(r),u^*(r),r) (y-u^*(r))$, we finally get
\begin{equation*}
\Big\langle \frac{\partial H}{\partial u} (\bar{q}^*(r),u^*(r),\bar{p}^\sigma(r),r) ,y-u^*(r) \Big\rangle_{\R^m} \leq 0.
\end{equation*}
Since this inequality holds for every $y \in \DS (u^*(r))$, we easily prove that it holds as well for every $v \in \Coad ( \DS (u^*(r)) )$. This proves \eqref{eqrscase}.

\paragraph{Right-dense points.} Let $s \in \LL_{[a,b[_\T} \cap \SS$ and $z \in \Omega$. Since $t\mapsto\langle \bar{w}_\amalg(t,u^*,\bar{q}^*_a),\bar{p}(t) \rangle_{\R^{n+1}}$ is an absolutely continuous function, the Leibniz formula \eqref{eqleibniz} yields $\langle \bar{p}(\cdot) , \bar{w}_\amalg(\cdot,u^*,\bar{q}^*_a) \rangle_{\R^{n+1}}^\DD = 0$ $\DD$-almost everywhere
on $[s,b[_\T$, and hence this function is constant on $[s,b]_\T$. It thus follows from \eqref{varineqdense} that
\begin{equation*}
\begin{split}
\langle \bar{p}(s) , \bar{w}_\amalg(s,u^*,\bar{q}^*_a) \rangle_{\R^{n+1}} 
& = \langle \bar{p}(b) , \bar{w}_\amalg(b,u^*,\bar{q}^*_a) \rangle_{\R^{n+1}} \\ 
& = \psi^0  w^{0}_\amalg (b,u^*,\bar{q}^*_a) +  \Big\langle \Big( \frac{\partial \bar{g}}{\partial q_2} (\bar{q}^*_a,\bar{q}^*(b)) \Big)^{\!\mathrm{T}}  \psi , w_\amalg (b,u^*,\bar{q}^*_a) \Big\rangle_{\R^n} \leq 0,
\end{split}
\end{equation*}
and since $\bar{w}_\amalg(s,u^*,\bar{q}^*_a) = \bar{f} (\bar{q}^*(s),z,s) - \bar{f} (\bar{q}(s),u^*(s),s)$, we finally get
$$
\langle  \bar{p}(s) , \bar{f} (\bar{q}^*(s),z,s) \rangle_{\R^{n+1}} \leq \langle \bar{p}(s) , \bar{f} (\bar{q}(s),u^*(s),s) \rangle_{\R^{n+1}}.
$$
Since this inequality holds for every $z \in \Omega$, the maximization condition \eqref{eqrdcase} follows.

\paragraph{Transversality conditions.} 
The transversality condition on the adjoint vector $p$ at the final time $b$ has been obtained by definition (note that $-\psi \in \O_\S (\bar{g} (\bar{q}^*_a,\bar{q}^*(b))) $ as mentioned previously). Let us now establish the transversality condition at the initial time $a$ (left-hand equality of \eqref{transv_cond}).
Let $\bar{q}_a \in \R^n \times \{ 0 \}$. With the same arguments as before, we prove that the function $t\mapsto\langle \bar{w}_{\bar{q}_a}(t,u^*,\bar{q}^*_a),\bar{p}(t) \rangle_{\R^{n+1}}$
is constant on $[a,b]_\T$. It thus follows from \eqref{varineqpointinit} that
\begin{equation*}
\begin{split}
\langle \bar{p}(a) , \bar{w}_{\bar{q}_a}(a,u^*,\bar{q}^*_a) \rangle_{\R^{n+1}} 
& = \langle \bar{p}(b) , \bar{w}_{\bar{q}_a}(b,u^*,\bar{q}^*_a) \rangle_{\R^{n+1}} \\ 
& = \psi^0  w^{0}_{\bar{q}_a} (b,u^*,\bar{q}^*_a)+ \Big\langle \Big( \frac{\partial \bar{g}}{\partial q_2}(\bar{q}^*_a,\bar{q}^*(b)) \Big)^{\!\mathrm{T}} \psi , w_{\bar{q}_a} (b,u^*,\bar{q}^*_a) \Big\rangle_{\R^n} \\ 
& \leq - \Big\langle \Big( \frac{\partial \bar{g}}{\partial q_1}(\bar{q}^*_a,\bar{q}^*(b)) \Big)^{\!\mathrm{T}} \psi, q_a \Big\rangle_{\R^n}  ,
\end{split}
\end{equation*}
and since $\bar{w}_{\bar{q}_a}(a,u^*,\bar{q}^*_a)=\bar{q}_a=(q_a,0)$, we finally get
\begin{equation*}
\Big\langle p(a,u^*,\bar{q}^*_a)+ \Big( \frac{\partial \bar{g}}{\partial q_1}(\bar{q}^*_a,\bar{q}^*(b)) \Big)^{\!\mathrm{T}} \psi , q_a \Big\rangle_{\R^n} \leq 0.
\end{equation*}
Since this inequality holds for every $\bar{q}_a \in \R^n \times \{ 0 \}$, the left-hand equality of \eqref{transv_cond} follows.

\paragraph{Free final time.}
Assume that the final time is not fixed in $\OCP$, and let $b^*$ be the final time associated with the optimal trajectory $q^*(\cdot)$. We assume moreover that $b^*$ belongs to the interior of $\T$ for the topology of $\R$. The proof of \eqref{transv_cond_time} then goes exactly as in the classical continuous-time case, and thus we do not provide any details. It suffices to consider variations of the final time $b$ in a neighbourhood of $b^*$, and to modify accordingly the functional of Section \ref{section41} to which Ekeland's Variational Principle is applied.

To derive \eqref{eqnullinthamilt}, we consider the change of variable $\tilde{t} = (t-a) / (b-a)$. The crucial remark is that, since it is an \textit{affine} change of variable, $\DD$-derivatives of compositions work in the time scale setting as in the time-continuous case. Then it suffices to consider the resulting optimal control problem as a parametrized one with parameter $b$ lying in a neighbourhood of $b^*$. Then \eqref{eqnullinthamilt} follows from the additional condition \eqref{condHlambda} of the PMP with parameters (see Remark \ref{remPMPparam}), which is established hereafter.

\paragraph{PMP with parameters (Remark \ref{remPMPparam}).}
To obtain the statement it suffices to apply the PMP to the control system associated with the dynamics defined by $\tilde{f}(\lambda,q,u,t) =  (f(\lambda,q,u,t),0)^{\!\mathrm{T}}$, with the extended state $\tilde{q}=(\lambda,q)$. In other words, we add to the control system the equation $\lambda^\Delta(t)=0$ (this is a standard method to derive a parametrized version of the PMP).
Applying the PMP then yields an adjoint vector $\tilde p = (p_\lambda,p)$, where $p$ clearly satisfies all conclusions of Theorem \ref{thmmain} (except \eqref{eqnullinthamilt}), and $p_\lambda^\Delta(t)=-\frac{\partial H}{\partial\lambda}(\lambda^*,q^*(t),u^*(t),p^\sigma(t),p^0,t)$ $\DD$-almost everywhere. From this last equation it follows that
$p_\lambda(b)-p_\lambda(a)=-\int_{[a,b^*[}\frac{\partial H}{\partial\lambda}(\lambda^*,q^*(t),u^*(t),p^\sigma(t),p^0,t)\, \Delta t$, and then \eqref{condHlambda} follows from the already established transversality conditions.

\bibliographystyle{plain}

\end{document}